\documentclass[leqno]{imsart}
\RequirePackage[OT1]{fontenc}
\RequirePackage{amsthm,amsmath,amsfonts}

\usepackage{color}
\usepackage{amsmath}
\usepackage{amssymb}
\usepackage{amsthm}
\usepackage{amsfonts}
\usepackage[utf8]{inputenc}
\numberwithin{equation}{section}

\usepackage[cm]{fullpage}
\startlocaldefs

\newcommand{\Q}{\mathbb{Q}}

\newcommand{\N}{\mathbb{N}}

\newcommand{\R}{\mathbb{R}}

\newcommand{\eps}{\epsilon}
\newcommand{\veps}{\varepsilon}
\newcommand{\cM}{\mathcal{M}}
\newcommand{\cD}{\mathcal{D}}
\newcommand{\cF}{\mathcal{F}}

\newcommand{\cL}{\mathcal{L}}
\newcommand{\cB}{\mathcal{B}}
\def\restrict#1{\raise-.5ex\hbox{\ensuremath|}_{#1}}

\newtheorem{theorem}{Theorem}[section]
\newtheorem{lemma}[theorem]{Lemma}
\newtheorem{proposition}[theorem]{Proposition}
\newtheorem{corollary}[theorem]{Corollary}

\newenvironment{remark}[1][Remark]{\begin{trivlist}
\item[\hskip \labelsep {\bfseries #1}]}{\end{trivlist}}
\newenvironment{conjecture}[1][Conjecture.]{\begin{trivlist}
\item[\hskip \labelsep {\bfseries #1}]}{\end{trivlist}}
\endlocaldefs

\begin{document}
\begin{frontmatter}

\title{On the boundary of the zero set of  super-Brownian motion and its local time.}
\runtitle{boundary of super-Brownian motion}

\begin{aug}
\author{\fnms{Thomas}
  \snm{Hughes}\thanksref{t1}\corref{}\ead[label=e1]{}} 
\author{\fnms{and Edwin} \snm{Perkins}\thanksref{t2}
\ead[label=e3]{hughes@math.ubc.ca, perkins@math.ubc.ca}}
\thankstext{t1}{Supported in part by an NSERC  CGS-D and an NSERC Discovery Grant.}
\thankstext{t2}{Supported in part by an NSERC Discovery Grant.}
\affiliation{
The University of British Columbia\thanksmark{t1} \thanksmark{t2}}


\address{Department of Mathematics\\
The University of British Columbia\\
1984 Mathematics Road\\
Vancouver, British Columbia V6T 1Z2\\
\printead{e3}}

\runauthor{Hughes and Perkins}
\end{aug}
\begin{abstract}
  If $X(t,x)$ is the density of one-dimensional super-Brownian motion, we prove that \break
  dim$(\partial\{x:X(t,x)>0\})=2-2\lambda_0\in(0,1)$ a.s. on $\{X_t\neq 0\}$, where $-\lambda_0\in(-1,-1/2)$ is the lead eigenvalue of a killed Ornstein-Uhlenbeck process.  This confirms a conjecture of Mueller, Mytnik and Perkins \cite{MMP17} who proved the above with positive probability.  To establish this result we derive some new basic properties of a recently introduced boundary local time (\cite{H2018}) and analyze the behaviour of $X(t,\cdot)$ near the upper edge of its support.  Numerical estimates of $\lambda_0$ suggest
  that the above Hausdorff dimension is approximately $.224$. 
  
   \end{abstract}




\end{frontmatter}

\section{Introduction and Statement of Results}
Let $(X_t,t\ge 0)$ denote a super-Brownian motion on the line starting at $X_0\neq 0$ under $P^X_{X_0}$. Here $X_0\in \cM_F(\R)$, the space of finite measures on $\R$ with the topology of weak convergence, and $P^X_{X_0}$ will denote any probability under which $X$ has the above law.  Our branching rate is chosen to be one so that the jointly continuous density, $X(t,x)$, of $X_t$ for $t>0$, is the unique in law solution of the stochastic partial differential equation (SPDE)
\begin{equation}\label {SPDE}
\frac{\partial X}{\partial t}(t,x)=\frac{1}{2}\frac{\partial^2 X}{\partial x^2}(t,x)+\sqrt{X(t,x)}\dot W(t,x),\ X\ge 0,\ X(0)=X_0
\end{equation}
(see Section III.4 of \cite{P2002}). Here $\dot W$ is a space-time white noise on $[0,\infty)\times\R$, and
the initial condition means that $X_t(dx)=X(t,x)dx\rightarrow X_0(dx)$ in $\cM_F(\R)$  as $t\downarrow 0$.  \\

The boundary of the zero set of $X_t$,
\begin{equation}\label{BZdef}
BZ_t=\partial\{x:X(t,x)=0\}=\partial\{x:X(t,x)>0\},
\end{equation}
was studied in \cite{MMP17}. The increased regularity of $X$ on and near this set has played an important role in the study of SPDE's such as \eqref{SPDE} (see \cite{MPS06} and \cite{MP11}). Mytnik and Perkins (unpublished) had obtained side conditions on $X$ which would give pathwise uniqueness in \eqref{SPDE} 
but which would imply that $\text{dim}(BZ_t)$, the Hausdorff dimension of $BZ_t$, is zero. The intuition here is that solutions to \eqref{SPDE} should only separate in their respective zero sets since these are the only points at which the noise coefficient is non-Lipschitz. So the smaller this set is, the harder it will be for solutions to separate.   In \cite{MMP17} it was shown that if $-\lambda_0\in(-1,-1/2)$ is the lead eigenvalue of the killed Ornstein-Uhlenbeck operator described below, then (see Theorem~1.3  of \cite{MMP17}) in fact
\begin{equation}\label{partialdimres}
P_{X_0}^X(\text{dim}(BZ_t)=2-2\lambda_0)>0.
\end{equation}
Here 
it was also conjectured (see the comment following Theorem~1.3 in \cite{MMP17}) that 
\begin{equation}\label{dimconj}
\text{dim}(BZ_t)=2-2\lambda_0\text{ a.s. on }\{X_t\neq0\}.
\end{equation}
In any case, the rigorous bounds on $\lambda_0$ mentioned above imply the dimension of $BZ_t$ is in $(0,1)$, at least with positive probability, and the aforementioned pathwise uniqueness problem remains unresolved in spite of a recent negative result in Chen \cite{Chen15}. Here pathwise non-uniqueness to \eqref{SPDE} was shown if an innocent looking immigration term of the form $\psi(x)$ ($\psi$ smooth, non-negative and compactly supported) is added to the right-hand side of \eqref{SPDE}.  The immigration term, however, gives $BZ_t$ positive 
Lebesgue measure and this is what allows Chen to establish separation of solutions.  \\

The boundary set itself is rather delicate as small perturbations of $X$ will of course completely change the nature of $BZ_t$. In particular, it is a non-monotone function of the initial condition.  This is one reason some of the standard zero-one arguments (see, e.g., the proof of Theorem~1.3 of \cite{P90} for the dimension of the range of $X$) were not able to resolve
the conjecture \eqref{dimconj}.  Our main result (Theorem~\eqref{dimwp1} below) will use a recently constructed boundary local time, $L_t(dx)$ of $BZ_t$
to confirm \eqref{dimconj}.  The local time was constructed by one of us (TH) in \cite{H2018}. It is a random
measure supported by $BZ_t$ which we are just beginning to understand, and some of its basic properties derived here will
play a central role in our arguments.  As a random measure supported on the set of points where solutions to \eqref{SPDE} can separate, $L_t(dx)$ has the potential of playing the same role in the study of SPDE's arising from
population models that ordinary local time does for stochastic differential equations.  Of course one would 
need to construct $L$  for a much larger class of random processes.  \\
In fact numerical estimates of $\lambda_0$ due to Peiyuan Zhu suggest that \eqref{dimconj} implies
\begin{equation}\label{dimest}
\dim(BZ_t)\approx .224\text{ a.s. on }\{X_t(1)>0\},
\end{equation}
perhaps larger than one may think given that $X(t,\cdot)$ is 
H\"older $1-\eta$ in space near its zero set for any $\eta>0$ (see Theorem~2.3 in \cite{MP11}).  
We briefly discuss this approximation below and give some evidence for the accuracy
of the estimate to the digits given.\\

It will often be more convenient to work with the canonical measure of super-Brownian motion, $\N_x$, which is a more fundamental object in many ways.  Recall that $X$ arises as the scaling limit of the empirical measures of critical branching random walk. $\N_x$ is a $\sigma$-finite measure on $C([0,\infty),\cM_F(\R))$ (the space of continuous measure-valued paths) describing the behaviour of the descendants of a single ancestor at $x$ at time $0$ (see Theorem II.7.3 of \cite{P2002}). A super-Brownian motion under $P^X_{X_0}$ may be constructed as the integral of a Poisson point process with intensity $\N_{X_0}(\cdot)=\int \N_x(\cdot)dX_0(x)$ (see \eqref{SBMPPPdecomp} below). In particular, if we write $X_t(\phi)=\int\phi(x)\,X_t(dx)$, then for $\phi\ge 0$,
\begin{equation}\label{PPPLT}
E^X_{X_0}(e^{-X_t(\phi)})=\exp\Bigl(-\int1-e^{-\nu_t(\phi)}d\N_{X_0}(\nu)\Bigr).
\end{equation}\\

Our next job is to describe $\lambda_0$ more carefully. We let
\begin{equation}\label{Fdef}
F(x)=-\log(P^X_{\delta_0}(X(1,x)=0))=\N_0(X(1,x)>0),
\end{equation}
where the last equality is a simple consequence of \eqref{PPPLT} with $\phi=\infty\delta_{x}$ and $X_0=\delta_0$ (see Proposition~3.3 of \cite{MMP17}).
Then $F$ is the unique positive symmetric $C^2$ solution to 
\begin{equation}\label{FODE}
\frac{F''}{2}(y)+\frac{y}{2}F'(y)+F(y)-\frac{F(y)^2}{2}=0,
\end{equation}
and
\begin{equation}\label{FBCS}
F'(0)=0,\quad \lim_{y\to\infty}y^2F(y)=0.
\end{equation}
(See (1.10),(1.12) of \cite{MMP17} and the discussion in Section~3 of the same reference.)  Let $Af(y)=\frac{f''(y)}{2}-\frac{yf'(y)}{2}$ be the generator of the Ornstein-Uhlenbeck process, $Y$, on the line.  For $\phi\in C([-\infty,\infty])$, the space of continuous functions on $\R$ with finite limits at $\pm \infty$, we let $A^\phi(f)=Af-\phi f$ be the generator of the Ornstein-Uhlenbeck process $Y^\phi$, now killed when $\int_0^t\phi(Y_s)\,ds$ exceeds an independent exponential mean one r.v.  If $m$ denotes the standard normal law on $\R$, the resolvent of $A^\phi$ is a Hilbert-Schmidt integral operator on the Hilbert space of square integrable functions with respect to $m$, $\cL^2(m)$. Therefore $A^\phi$ has a complete orthonormal system of eigenfunctions $\{\psi_n^\phi:n\ge 0\}$ with non-positive eigenvalues $\{-\lambda_n^\phi\}$ ordered so that $-\lambda_n^\phi$ decreases to $-\infty$.   The lead eigenvalue $-\lambda_0^\phi\le 0$ is simple and so 
has a unique normalized eigenfunction $\psi_0^\phi$.  See Theorem~\ref{thm_OU} below for this and related information. If we set $\phi=F$, then our eigenvalue $-\lambda_0$ is $-\lambda^F_0$ which is in $(-1,-1/2)$ by an elementary calculation in Proposition 3.4(b) of \cite{MMP17}, using the fact (Proposition 3.4(b) of \cite{MMP17}) that 
\begin{equation}\label{lambdahalfF}
\lambda_0^{F/2}=1/2.
\end{equation}
Here then is our main result.
\begin{theorem}\label{dimwp1} For any $X_0\in\cM_F(\R)\setminus\{0\}$ and $t>0$,
\begin{equation}
\text{dim}(BZ_t) = 2 - 2\lambda_0\in (0,1)\quad P_{X_0}^X-\text{a.s. and }\N_0-\text{a.e. on }\{X_t > 0 \}.
\end{equation}
\end{theorem}

In fact Theorem~1.3(a) in \cite{MMP17} already gives 
\begin{equation}\label{dimupperbound}
\text{dim}(BZ_t)\le 2-2\lambda_0\ P^X_{X_0}-\text{a.s.  and }\N_0-\text{a.e.}
\end{equation}
Although the above reference only considers $P^X_{X_0}$, the result for $\N_0$ then follows easily by the Poisson point process decomposition mentioned above (see (2.5) below), just as in the last six lines of the proof of Theorem~\ref{thm_01-law} at the end of Section~\ref{secthm01}.  Therefore it is the lower bound on $\text{dim}(BZ_t)$ that we must consider.  The lower bound on the dimension was attained with positive probability in Theorem~5.5 of \cite{MMP17} by first deriving a sufficient capacity condition for $BZ_t$ to intersect a given set, $A$, with positive probabiity (Theorem~5.2 of \cite{MMP17}) and then taking $A$ to be the range of an appropriate L\'evy process.  As was already noted, the authors were unable to use this approach to establish the lower bound a.s.  The standard approach to lower bounds on Hausdorff dimension is through the energy method.  That is, first construct a finite random measure or local time, $L_t$, supported by $BZ_t$ such that 
\begin{equation}\label{energyest}
E\Bigl(\iint|x-y|^{-\alpha}dL_t(x)dL_t(y)\Bigr)<\infty\quad\forall \ 0<\alpha<2-2\lambda_0.
\end{equation}
The energy method (see Theorem 4.27 of \cite{MP10}) would then imply 
\begin{equation}\label{dimlowerbound}
\text{dim}(BZ_t)\ge 2-2\lambda_0 \text{ a.s. on }\{L_t\neq 0\}.
\end{equation}
The existence of such a boundary local time was established in \cite{H2018}, confirming
a construction conjectured in Section~5 of \cite{MMP17}, which we briefly describe now.
Define a measure $L^\lambda_t \in \cM_F(\R)$ by
\begin{equation} \label{def_Llambda}
L^\lambda_t(\phi) = \int \phi(x) \, \lambda^{2\lambda_0} X(t,x) e^{-\lambda X(t,x)} \,dx
\end{equation}
for bounded Borel functions $\phi$. Note that as $\lambda$ gets large $L^\lambda_t$ becomes concentrated
on the set of points $x$ where\hfil\break
 $0<X(t,x)=O(1/\lambda)$. The normalization of $\lambda^{2\lambda_0}$ comes from the left tail behaviour of $X(t,x)$ in Theorem~1.2 of \cite{MMP17}. The following result is taken from \cite{H2018}, more specifically it is included in Theorems~1.1, 1.2, 1.3 and 1.5, and Proposition~1.6 of \cite{H2018}.
\medskip

\textbf{Theorem A.} \emph{
(a) There is a finite atomless random measure, $L_t$, on the line such that under $\N_0$ or $P^X_{X_0}$, 
\[L^\lambda_t \to L_t  \text{ in measure in the metric space $\cM_F(\R)$ as }\lambda\to\infty. \]
Moreover $L_t$ is supported on $BZ_t$ a.s.\\
(b) There is a positive constant $C_A$  such that 
for any Borel $\phi:\R\to[0,\infty)$,
\begin{equation}\label{Lmeanmeasure}\int L_t(\phi)d\N_0=C_At^{-\lambda_0}\int\phi(\sqrt t z) \psi_0^F(z)dm(z).\end{equation}
(c) \eqref{energyest} holds under both $\N_0$ and $P^X_{X_0}$. \\
(d) There is a constant $C_B$ such that 
\begin{equation}\label{Lsecondmoment}\int L_t(1)^2 d\N_0\le C_Bt^{1-2\lambda_0}.
\end{equation}}

Let $S(X_t)=\overline{\{x:X(t,x)>0\}}$ be the closed support of $X_t$ and define $U_t=\sup(S(X_t))$ to be the upper most point of the support. 
It now follows from Theorem A that \eqref{dimlowerbound} holds under both $P^X_{X_0}$ and $\N_0$ (for the latter one can work under the probability $\N_0(\cdot|X_t\neq0)$). And so Theorem~\ref{dimwp1} is 
immediate from \eqref{dimupperbound} and the following:
\begin{theorem} \label{thm_01-law}
Under the measures $\N_0$ and $P^X_{X_0}$, $L_t > 0$ almost surely on $\{X_t > 0 \}$. In fact, almost surely on $\{X_t > 0\}$, $L_t((U_t - \delta, U_t)) > 0 $ for all $\delta > 0$.
\end{theorem}
This theorem shows that as long as $X_t$ has not gone extinct, the part of $BZ_t$ at its upper edge will have positive $L_t$ measure, and, in particular, $L_t$ itself is not equal to the zero measure. It is natural to consider a local version of the above and show that $L_t$ will charge any open interval which contains points in $BZ_t$.  
This clearly fails (note from Theorem~A that $L_t$ is atomless) if $X_t(\cdot)$ has isolated zeros, which
clearly would be in $BZ_t$.  
An elementary argument shows that $\partial S(X_t) \subseteq BZ_t$ and the former set clearly will not
contain isolated zeros of $X_t(\cdot)$.  Given that the existence of isolated zeros of $X_t(\cdot)$ remains
unresolved (we conjecture that they do not exist), here then is our local version of Theorem~\ref{thm_01-law}:

\begin{theorem} \label{thm_01_law_local}
For $t>0$, $P_{X_0}^X$ and $\N_0$ a.s.,  for any $a<b$, $(a,b) \cap \partial S(X_t) \neq \emptyset \text{ implies } L_t((a,b)) > 0$.
\end{theorem}

Evidently we do not know whether or not $BZ_t\setminus \partial S(X_t)$ is non-empty; isolated zeros are not the only possible points in this set--see Lemma~\ref{lemma_topsupport} below.
Nonetheless 
we make the following conjecture:
\begin{conjecture}\label{suppconj} $L_t$ is supported on $\partial S(X_t)$ and so dim$(\partial S(X_t))=2-2\lambda_0$ on $\{X_t\neq 0\}$ $P_{X_0}^X$-a.s. and $\N_0$-a.e.,
\end{conjecture}
the last conclusion being immediate from the first by \eqref{dimupperbound}, Theorem~\ref{thm_01-law}, Theorem~A(a,b) and
the energy method described above.\\

\begin{corollary}
For $t>0$, $P_{X_0}^X$ and $\N_0$ a.s., for any $a<b$, $(a,b) \cap \partial S(X_t) \neq \emptyset \text{ implies } \text{dim}(BZ_t \cap (a,b)) = 2-2\lambda_0$.
\end{corollary}
\begin{proof} By considering rational values we may fix $a$ and $b$ and work under either $P^X_{X_0}$ or $\N_0(\cdot|X_t\neq 0)$.  Assume $(a,b)\cap\partial S(X_t)\neq\emptyset$.   In view of \eqref{energyest} we may apply the energy method to $L_t|_{(a,b)}$, which is a.s. non-zero by Theorem~\ref{thm_01_law_local}, and so conclude that dim$(BZ_t\cap(a,b))\ge 2-2\lambda_0$ a.s.  on $\{(a,b)\cap \partial (S(X_t))\neq 0\}$.  The corresponding upper bound is immediate from \eqref{dimupperbound}.
\end{proof}

We comment briefly on the numerical approximation of $\lambda_0$ carried out by Peiyuan Zhu in \cite{Z17}.
One first needs to numerically approximate $F$ using an an ODE solver and the ``shooting method'' to find  the minimal
value of $c$ so that $F_c(0)=c$, $F_c'(0)=0$ and $F_c$ satisfying \eqref{FODE} remains non-negative. It is known that $F_c=F$ (see, e.g., \cite{BPT86}). One then approximates this numerically generated $F$ by a linear combination of Gaussians $\hat F$ (with varying means and variances).   We estimate $-\lambda_0^F$ by $-\lambda_0^{\hat F}$, the lead eigenvalue of the Ornstein-Uhlenbeck operator with $\hat F$-killing on
a large interval $[0,K]$ with Neumann boundary conditions. $K$ must be taken sufficiently large to
approximate the corresponding operator on $[0,\infty)$. The final step is then to use CHEBFUN software to estimate $\lambda_0^{\hat F}$. One could also obtain $\hat F$ by interpolating between the numerically generated grid points using Chebychev polynomials--the results agree to the given accuracy.  We have some faith in the resulting approximation
of $\lambda_0^F\approx .8882$ because  if we replace $F$ with $F/2$, the same method leads to $\lambda_0^{F/2}\approx.5000$. This compares well with the exact (known) value in \eqref{lambdahalfF}. 
\\

The proof of Theorem~\ref{thm_01-law} includes some input from the semilinear pde's associated
with super-Brownian motion (such as \eqref{slpde} below) which are carried out in Section~\ref{secpde}. This is then used in Section~\ref{secthm01}, to study $X_t(dx)$ near the upper end of its support, $U_t$.
For $\epsilon > 0$, define \begin{equation} \label{def_taueps}
\tau^\epsilon =\tau^\epsilon(t)= \inf \{ x \in \R : X_t([x, \infty)) < \epsilon\}.
\end{equation}
In particular, if $X_t(1) < \epsilon$, then $\tau^\epsilon = -\infty$. The following result gives 
some insight into the behaviour of $X_t$ near the upper edge of its support and so the following first moment bound, which is proved in Section~\ref{secthm01}, may be of independent interest.

\begin{proposition} \label{prop_boundaryGrowthBd}
There is a non-increasing function, $c_{\ref{prop_boundaryGrowthBd}}(t)$, such that for all $t,\epsilon>0$ and $u> 0$:\hfil\break
(a) For any $X_0\in \cM_F(\R)$, $E^X_{X_0}\left(\int_{\tau^\eps(t)-u}^\infty X_t(x)dx\right)\le c_{\ref{prop_boundaryGrowthBd}}(t)X_0(1)(u^2\vee\eps)$.

(b) 
$\N_0\left( \int_{\tau^\epsilon(t) - u}^\infty X_t(dx) \right) 
\leq c_{\ref{prop_boundaryGrowthBd}}(t)(u^2 \vee \epsilon)$.
\end{proposition}

One can understand the important $u^2$ behaviour in the above for small $u,\epsilon$ from the improved modulus of continuity
of $X(t,\cdot)$ near its zero set (mentioned above).  Theorem~2.3 of \cite{MP11} shows that for $\eta>0$ there is  $\delta(\omega)>0$ so that $|X(t,x)-X(t,x+h)|\le |h|^{1-\eta}$ for $X(t,x)\le |h|\le \delta(\omega)$.
This readily leads to (for $\epsilon,\ u$ small) $X(t,\tau^\epsilon(t))\le \epsilon^{.5-\eta}$ and after a short
argument (consider $u\ge \epsilon^{.5-\eta}$ and $u< \epsilon^{.5-\eta}$ separately) that 
\[\int_{\tau^\epsilon(t)-u}^{U_t} X(t,x)dx=\int_{\tau^\epsilon(t)-u}^{\tau^\epsilon}X(t,x)\,dx+\epsilon\le c(\epsilon^{1-2\eta}+u^{2-\eta}),\]
which comes close to the above mean behaviour. 
The actual proof uses the unique non-negative solution, $v_t^\infty(x)=v^\infty(t,x)$, in $C^{1,2}((0,\infty)\times\R)$ of 
\begin{equation}\label{slpde}
\frac{\partial v^\infty}{\partial t}(t,x)=\frac{1}{2}\frac{\partial^2 v^\infty}{\partial x^2}-\frac{(v^\infty)^2}{2},\quad v^\infty_0=\infty1_{(-\infty,0]}.
\end{equation}
Such semilinear parabolic equations arise of course as  exponential dual functions for super-Brownian motion--see Section~\ref{secpde} for more on this in general, and Theorem~\ref{vinfinitypde} for more information on the particular equation above, including its precise meaning.  More specifically, the proof uses $G(x)=v^\infty(1,x)$ which also is the unique $C^\infty$
solution of \eqref{FODE} but now with the boundary conditions (see Lemma~\ref{lemma_Gproperties}(c))
\begin{equation}\label{GBCs}
\lim_{x\to \infty}x^2 G(x)=0,\quad \lim_{x\to -\infty} G(x)=2.
\end{equation}
Using a Palm measure formula for $X_t$ (Theorem 4.1.3 from \cite{DP}), the Feynman-Kac Formula and some pde bounds (notably Proposition~\ref{prop_vlambdaROC}), we show (see \eqref{e_endmassprop2}) that for $u^2\ge \epsilon$ (from which the general
case follows easily),
\begin{equation}\label{evalueupperbound}
\N_0\Bigl(\int_{\tau^\epsilon(t)-u}^\infty X(t,x)\,dx\Bigr)\le c(t)E^Y_m\Bigl(\exp\Bigl(-\int_0^{\log(1/u^2)}G(Y_s)\,ds\Bigr)\Bigr),
\end{equation}
where $Y$ is an (unkilled) Ornstein-Uhlenbeck process with initial law $m$ under $P^Y_m$.  So, as in \cite{MMP17}, one can use the spectral decomposition of $A^G$ to see that the  right-hand side of \eqref{evalueupperbound} is at most $c(t)e^{-\lambda_0^G\log(1/u^2)}=c(t)u^{2\lambda_0^G}$. Unlike $\lambda_0^F$, we can identify the eigenfunction for $\lambda_0^G$ and verify that $\lambda_0^G=1$ (Proposition~\ref{prop_leadeig}), and hence obtain the required bound in Proposition~\ref{prop_boundaryGrowthBd}(b). \\

Turning to Theorem~\ref{thm_01-law} itself, Theorem~A(b),(d) and the second moment method easily give  (Lemma~\ref{lemma_Ltposhalfline})
\begin{equation*}
\N_0(L_\delta([3\sqrt{\delta},\infty))>0|X_\delta\neq 0)\ge p>0\quad\forall\ \delta>0.
\end{equation*}
One can then use this to conclude that the right-most ancestor, say at $x$, at time $t-\delta$ of the population at time $t$
will have descendants at time $t$ with a positive boundary local time on $[x+3\sqrt{\delta},\infty)$ with conditional (on $\cF_{t-\delta}$) probability at least $p$. Now one must show that the descendants of the  other ancestors at time $t-\delta$ do 
not flood into the boundary region of the right-most ancestor and hence remove it from the overall boundary.
This issue captures the delicate and non-monotone character of the boundary.  To resolve it we use a classical hitting estimate for $X$ from \cite{DIP89} (see Theorem~\ref{th_DIP33} below) and Proposition~\ref{prop_boundaryGrowthBd}. This will lead to a uniform lower bound on $P^X_{X_0}(L_t>0|\cF_{t-\delta_n})$ with high probability at least on $\{X_t\neq 0\}$ and the martingale convergence theorem
then shows $L_t>0$ with high probability on $\{X_t\neq 0\}$.  \\

Theorem~\ref{thm_01_law_local} is proved in Section~\ref{sec:local01}. Section~\ref{prelims} reviews a number of standard tools we will need in the proofs including the spectral decomposition of the killed Ornstein-Uhlenbeck processes, some cluster decompositions of super-Brownian motion based on historical information, and the aforementioned hitting estimate for super-Brownian motion.  \\

{\bf Acknowledgement.} We thank Peiyuan Zhu for allowing us to report on his numerical work on the 
estimation of $\lambda_0$.

\section{Some Preliminaries}\label{prelims}
\subsection{Killed Ornstein-Uhlenbeck Processes}\label{KOU}
Recall that $Y$ is an Ornstein-Uhlenbeck process with generator $A$, starting at $x$ under $P^Y_x$.  
As above for $\phi\in C([-\infty,\infty])$, $\phi\ge 0$, $A^\phi$ is the generator of the Ornstein-Uhlenbeck process, $Y^\phi$, killed at time 
$\rho_\phi=\inf\{t:\int _0^t\phi(Y_s)ds>e\}$, where $e$ denotes an independent exponential r.v. with mean one.   
The result below is standard, and included in Theorem~2.3 of Mueller, Mytnik and Perkins \cite{MMP17}. 

\begin{theorem} \label{thm_OU}
(a) $A^\phi$ has a complete orthonormal family $\{\psi_n : n \geq 1 \}$ of $C^2$ eigenfunctions of $\cL^2(m)$ satisfying $A^\phi \psi_n = -\lambda_n \psi_n$, where $\{-\lambda_n\}_{n=1}^\infty$ is a non-increasing sequence of non-positive eigenvalues such that $ \lambda_n \to \infty$. Furthermore, $-\lambda_0$ is a simple eigenvalue and $\psi_0 > 0$.\\
(b) Let $\theta = \int \psi_0\, dm$. For all $0<\delta$, there exists $c_\delta$ such that for all $x \in \R$,
\begin{equation} \label{OU_survivalprob}
|e^{\lambda_0 t}P^Y_x (\rho_\phi > t) - \theta \psi_0(x)|  \leq c_\delta e^{\delta x^2} e^{-(\lambda_1 - \lambda_0) t},
\end{equation}
and
\begin{equation} \label{OU_psi0bd}
\psi_0(x) \leq c_\delta e^{\delta x^2}.
\end{equation}
In particular, 
\begin{equation} \label{OU_survivalprobbd}
P^Y_x (\rho^\phi > t) \leq C_{\delta} e^{\delta x^2} e^{-\lambda_0 t}.
\end{equation}
and 
\begin{equation}\label{OUSbound}P^Y_m( \rho^\phi > t) \leq Ce^{-\lambda_0 t}.
\end{equation}
\end{theorem}\subsection{Cluster and Historical Decompositions of Super-Brownian Motion}\label{SBMdecomps}

We recall the cluster decomposition of super-Brownian motion from Theorem 4 in Section IV.3 of \cite{LG1999}. 
If $X_0\in \cM_F(\R)$, let $\Xi_{X_0}$ be a Poisson point process on the space $C([0,\infty), \cM_F(\R))$
of continuous measure-valued paths with intensity $\N_{X_0}(\cdot)=\int\N_x(\cdot)dX_0(x)$. Then
\begin{equation}\label{SBMPPPdecomp}
X_t(\cdot)=\begin{cases}\int\nu_t(\cdot)\Xi_{X_0}d(\nu)& \text{if $t>0$}\\
X_0(\cdot)& \text{if $t=0$}
\end{cases}
\end{equation}
defines a super-Brownian motion with initial state $X_0$.  In particular this shows that for $t>0$, ($\stackrel{\cD}{=}$ denotes equality in law)
\begin{equation} \label{e_clusterdecomp1}
X_t \stackrel{\cD}{=} \sum_{i =1}^{N} X^i_{t},
\end{equation}
where $N$ has a Poisson law with mean $2X_0(1)/t=\N_{X_0}(X_t>0)$, and given $N$, $\{X^i_{t} : i \leq N\}$ are iid random measures with law $\N_{X_0}(X_t\in\cdot|X_t>0)$.  The summands in \eqref{e_clusterdecomp1} correspond to the contributions to $X_t$ from each of the finite number of ancestors
at time $0$ of the population at time $t$. \\

We will also make use of the historical process associated with a super-Brownian motion. The historical process encodes the genealogical information of the super-Brownian motion $X$. Good introductions may be found in \cite{DP}, or Sections II.8 and III.1 of \cite{P2002}. Let $C([0,\infty), \R)$ denote the space of continuous $\R$-valued paths on $[0,\infty)$, endowed with the compact-open topology. The historical process $(H_t : t\geq 0)$ is a measure-valued time-inhomogeneous Markov process taking values in $\cM_F(C([0,\infty), \R))$ such that $y(\cdot)=y(t\wedge\cdot)$ for $H_t$-a.a. $y$ for all $t\ge 0$ a.s. 
If we identify constant paths with $\R$, then, viewing $H_0$ as an element of $\cM_F(\R)$, we 
can recover the super-Brownian motion $X$ starting at $X_0=H_0$ from its associated historical process $H$ by projecting
$H_t$ onto time $t$, that is,
$X_t(\cdot)=H_t(\{y\in C([0,\infty), \R):y(t)\in\cdot\})$. Intuitively, $(y(s),s\le t)$ gives the historical path of the particle $y(t)$ in
the support of $X_t$. We will use a modulus of continuity for the paths $y$ governed by $H_t$. Let $S(H_t)$ denote the closed support of $H_t$ and set $h(r) = (r \log(1/r))^{1/2}$. For $c>0$ and $\delta > 0$, define $K(c,\delta)$ by
\begin{equation}\label{Kdef}
K(c,\delta) = \{ y \in C([0,\infty), \R) : | y_r - y_s | \leq ch(r-s) \,\,\, \forall \, r,s \geq 0 \text{ s.t. } |r-s| \leq \delta \}. 
\end{equation} 
By Theorem III.1.3(a) of \cite{P2002}, if $c>2$ and $T>0$, then $P^X_{X_0}$-a.a. $\omega$, there exists $\delta = \delta(T,c,\omega) > 0$ a.s. such that 
\begin{equation}\label{Hmod}S(H_t) \subset K(c,\delta)\text{ for all } t\in[0,T]. 
\end{equation}
Moreover the proof of the above shows that for any $c>2, T>0$ there are $\rho(c)>0$ and $C(T)$ such
that 
\begin{equation}\label{deltarate}
P^X_{X_0}(\delta(T,c)\le r)\le C_{\ref{deltarate}}(T)r^{\rho(c)} \text{ for all } r\in(0,1],
\end{equation}
where
\begin{equation}\label{rhoasymp}
\lim_{c\to\infty}\rho(c)=\infty.
\end{equation}\\

A second decomposition of a superprocess based on historical information will also play an important role in our arguments.  Let $(\cF_t)$ be the usual right-continuous completed filtration generated by $H$ and assume $0\le\delta\le t$ are fixed.   Assume $\tau\in[-\infty,\infty]$ is a $\sigma(X_{t-\delta})$-measurable random variable.  We decompose $X_{t-\delta}$ into the sum of two random measures:
\begin{equation}\label{XRdef}
X^R_{t-\delta}(dx)=X^{R,\tau,\delta}_{t-\delta}(dx)=1_{\{x\ge\tau\}}X_{t-\delta}(dx)\ \text{ and }\ X^L_{t-\delta}(dx)=X^{L,\tau,\delta}_{t-\delta}(dx)=1_{\{x<\tau\}}X_{t-\delta}(dx).\end{equation}
We then track the descendants of each of these populations at future times and so define measure-valued processes by
\[\hat X^R_s(\phi)=\hat X^{R,\tau,\delta}_s(\phi)=\int\phi(y_{t-\delta+s})1(y_{t-\delta}\ge\tau)H_{t-\delta+s}(dy);\ \ \hat X^L_s(\phi)=\hat X^{L,\tau,\delta}_s(\phi)=\int\phi(y_{t-\delta+s})1(y_{t-\delta}<\tau)H_{t-\delta+s}(dy).\]
Clearly we have 
\begin{equation}\label{histdecomp}
\hat X^R_s+\hat X^L_s=X_{t-\delta+s}\text{ for all }s\ge 0,\ \text{ and  (if $s=0$) }\ X^R_{t-\delta}+X^L_{t-\delta}=X_{t-\delta}.
\end{equation}
By (III.1.3) on p. 193 of \cite{P2002} and
the Markov property of $H$, we get:
\begin{align}\label{histcondlaw}
&\text{Conditional on $\cF_{t-\delta}$, $(\hat{X}^{R}_s)$ and $(\hat{X}^{L}_s)$ are independent $(\cF_{t-\delta+s})$-super-Brownian motions with initial laws}\\
&\text{$X^R_{t-\delta}$ and $X^L_{t-\delta}$, respectively.}\nonumber
\end{align}

Given the above decompositions of super-Brownian motion into a sum of independent super-Brownian motions, it is not surprising that we will also need to know how the corresponding boundary local time, $L_t$, decomposes.  Recall that a sum
of $n$ independent super-Brownian motions with initial conditions $X_0^1,\dots, X_0^n$ is a super-Brownian motion starting at 
$X_0^1+\dots+X_0^n$.  The next result is Theorem~1.9 of \cite{H2018}.

\begin{theorem}\label{THMW} Suppose $X^1,\dots,X^n$ are independent one-dimensional super-Brownian motions, starting at 
$X^1_0,\dots, X_0^n\in\cM_F(\R)$ and with boundary local times $L^1_t,\dots L^n_t$.  Let $X=\sum_{i=1}^n X^i$ and $L_t$ be the boundary local time of $X$. Then
\begin{equation}\label{LTdecompo}
dL_t(x)=\sum_{i=1}^n1\Bigl(\sum_{j\neq i}X^j(t,x)=0\Bigr)dL_t^i(x)=\sum_{i=1}^n1(X(t,x)=0)dL_t^i(x).
\end{equation}
\end{theorem}

\subsection{Hitting Probabilities of Super-Brownian motion}
The proofs of our main theorems will make use of bounds on hitting probabilities for super-Brownian motion.
\begin{theorem}  \label{th_DIP33} There exists a universal constant $c_{\ref{th_DIP33}} < \infty$ such that: \\
(i) For $R > 2\sqrt t$,
\[ \N_0 \big( X_s([R,\infty)) > 0 \,\, \text{ for some } s \leq t \big) \leq c_{\ref{th_DIP33}}\, R^{-2} \left( \frac{R}{\sqrt t}\right)^{3} e^{-R^2/2t}. \] 
(ii) For all $X_0 \in \cM_F(\R)$ such that $X_0$ is supported on $(-\infty,0]$ and for all $R>2\sqrt{t}$, we have
\[ P^X_{X_0} \big( X_s([R,\infty)) = 0 \,\, \text{ for all } s \leq t \big) \geq \exp \left( -c_{\ref{th_DIP33}}\, \int_{-\infty}^0 (R-x)^{-2} \left( \frac{R-x}{\sqrt t}\right)^{3} e^{-(R-x)^2/2t} dX_0(x) \right). \]
\end{theorem} 
\begin{proof}
(i) is a simple consequence of Theorem 3.3(b) of \cite{DIP89} with $d=1$ (and its proof) and \eqref{SBMPPPdecomp}. \\

We derive (ii) as a consequence of (i) by using \eqref{SBMPPPdecomp}. Indeed, this result and well-known formulas for the Laplace transform of a Poisson point process (see, for example, Theorem 24.14 of \cite{Klenke}) imply
that for  $R > 2\sqrt t$ and $\theta>0$, we have
\begin{align}
E^X_{X_0} \left( \exp \left( - \theta \int_0^t X_s([R,\infty))\,ds \right) \right) = \exp \left(- \int N_{x} \left( 1- \exp \left( - \theta \int_0^t X_s([R,\infty)) \, ds\right) \right) dX_0(x) \right). \nonumber
\end{align}
A simple application of Dominated Convergence allows us to let $\theta \to \infty$ and conclude that
\begin{align}
P^X_{X_0} \big( X_s([R,\infty)) = 0 \,\, \text{ for all } \, s < t  \big) =  \exp \left(- \int \N_{x} \big( X_s([R,\infty)) > 0 \,\, \text{ for some } s \leq t \big)\, dX_0(x) \right). \nonumber
\end{align}
Part (ii) follows by applying (i) and translation invariance.
\end{proof}

\section{Some Semi-linear Partial Differential Equations}\label{secpde}
We recall the relationship of the Laplace functional of super-Brownian motion with solutions of a semi-linear partial differential equation (PDE). We first present the integral form of the equation. Let $\cB_{b+}(\R)$ denote the space of non-negative bounded Borel functions on the line. Let $E^B_x$ denote the expectation of standard Brownian motion with $B_0 = x$, and denote the Brownian semigroup by $S_t$, ie. $S_t \phi(x) = E_x^B(\phi(B_t))$. By Theorem II.5.11 of \cite{P2002}, for $\phi \in \cB_{b+}(\R)$ there exists a unique non-negative solution to the integral equation
\begin{align} \label{e_integraleq}
v_t = S_t \phi - \int_0^t S_{t-s}(v_{s}^2/2) \,ds \hspace{ 4 mm } \text{ for } (t,x) \in [0,\infty) \times \R,
\end{align}
which we denote by $V_t^\phi(x)$, such that for all $X_0 \in  \cM_F(\R)$,
\begin{equation} \label{e_LapFun}
E^X_{X_0} \big( e^{-X_t(\phi)} \big) = e^{-X_0(V^\phi_t)}.
\end{equation}
It follows from \eqref{SBMPPPdecomp} and the above with $X_0=\delta_x$ that
\begin{equation} \label{e_LapFunCanon}
\N_x \big( 1 - e^{-X_t(\phi)} \big) = V^\phi_t(x).
\end{equation}
It is clear from \eqref{e_integraleq} that 
$ V^\phi_t(x)\le S_t\phi(x)\le \Vert\phi\Vert_\infty$, and so $V^\phi_t(x)-S_t\phi(x)\to 0$ as $t\downarrow 0$ pointwise in $x$.  This readily implies that 
\[V^\phi_t\stackrel{v}{\rightarrow} \phi=V_0^\phi\text{ as }t\downarrow 0,\]
where $\stackrel{v}{\rightarrow}$ denotes vague convergence of the Radon measure $V^\phi_t(x)dx$ to $\phi(x)dx$. 
\eqref{e_integraleq} is known as the mild form of the PDE
\begin{equation} \label{e_dualPDE}
\frac{\partial v_t}{\partial t} = \frac 1 2 \frac{\partial^2 v_t}{\partial x^2} - \frac{v_t^2}{2} \hspace{4 mm} \text{ for } (t,x) \in (0,\infty) \times \R,  \hspace{4 mm} v_t \stackrel{v}{\rightarrow}\phi=v_0 \text{  as } t\downarrow 0,
\end{equation}
where it will be understood that solutions of \eqref{e_dualPDE} will be in the space $C^{1,2}((0,\infty)\times\R)$ of functions with continuous partial derivatives up to order $1$ in time and $2$ in space on the given open set. This formulation allows one to consider initial conditions which are measures.  In this 
context Marcus and V\'eron \cite{MV1999} (Theorem 3.5) proved existence and uniqueness of a (non-negative) solution, $\bar V^\phi$,
to \eqref{e_dualPDE} as a rather special case of more general initial conditions which they classify with their initial trace theory. The use of their general theory may seem like overkill, but it will soon be convenient to use
a stability result in  \cite{MV1999}. It is easy to show that their solutions also satisfy the mild form \eqref{e_integraleq} as we now sketch.  First, monotonicity of $\bar V^\phi$ in $\phi$ (e.g. Theorem~3.4 of  \cite{MV1999})
and comparison with the elementary solution with initial (constant) value $\Vert\phi\Vert_\infty$ show that  
\begin{equation}\label{Vubound}
\bar V^\phi(t,x)\le \Vert\phi\Vert_\infty.
\end{equation}
For $\veps>0$, $\bar V^{\veps,\phi}_t:=\bar V^\phi_{t+\veps}$ defines the unique solution to \eqref{e_dualPDE} with $C^2$ initial data $\bar V^\phi_\veps$ and evidently the solution is now in $C^{1,2}([0,\infty)\times\R)$.  Such strong solutions are known to be solutions of the mild equation \eqref{e_integraleq} (see, e.g., the outline following Proposition II.5.10 in \cite{P2002} and use the above boundedness).  We  therefore have
\[ \bar{V}^\phi_{t+\epsilon} = S_t \bar{V}^\phi_\epsilon - \int_0^t S_{t-s}(\bar{V}^{\phi\,\,\, 2}_{s+ \epsilon}/2) \,ds \hspace{ 4 mm } \text{ for } (t,x) \in [0,\infty) \times \R.\] 
It is easy to justify taking the limit pointwise as $\epsilon \downarrow 0$ (use \eqref{Vubound}), which shows that $\bar{V}^\phi_t$ solves the integral equation \eqref{e_integraleq}. By uniqueness of solutions to \eqref{e_integraleq}
we conclude that $\bar V_t^\phi=V_t^\phi$. We therefore have that for $\phi \in \cB_{b+}(\R)$, there exists a unique non-negative solution $V^\phi_t$ to \eqref{e_dualPDE} (also satisfying \eqref{e_integraleq}) such that \eqref{e_LapFun} and \eqref{e_LapFunCanon} hold.\\

For $\lambda>0$, we denote by $v^\lambda_t$ the unique non-negative solution of
\begin{align}\label{e_dualPDElambda2}
&\frac{\partial v_t}{\partial t} = \frac 1 2 \frac{\partial^2 v_t}{\partial x^2} - \frac{v^2_t}{2}, \hspace{8 mm} \text{ for } (t,x) \in (0,\infty) \times \R,  \hspace{4 mm} v_t \stackrel{v}{\rightarrow} \lambda 1_{(-\infty,0]} =v_0^\lambda\text{  as } t\downarrow 0.
\end{align}
Given the above discussion, $v^\lambda_t$ is also a solution of \eqref{e_integraleq} with $\phi = \lambda 1_{(-\infty,0]}$. We will sometimes write $v^\lambda_t(x) = v^\lambda(t,x)$. By (\ref{e_LapFunCanon}), translation invariance and symmetry, $v^\lambda_t$ satisfies for all $t>0$,
\begin{align} \label{e_LapFunCanonlambda2}
v^\lambda_t(x) = \N_0 \big(1 - e^{-\lambda X_t((-\infty,-x])} \big) = \N_0 \big(1 - e^{-\lambda X_t([x,\infty))} \big). 
\end{align}
Similarly, by \eqref{e_LapFun} we also have for $t>0$,
\begin{equation} \label{e_LapFunlambda2}
E^X_{\delta_0}\left( e^{-\lambda X_t([x,\infty))} \right) = e^{-v^\lambda_t(x)}.
\end{equation}
It is an exercise to use uniqueness in  (\ref{e_dualPDElambda2}) or scaling properties of super-Brownian motion to show that $v^\lambda$ satisfies the following scaling relationship:
\begin{equation} \label{e_vscale}
v^\lambda(t,x) = rv^{\lambda/r}(rt,\sqrt{r}x)\hspace{ 5 mm} \forall \, r,\lambda > 0.
\end{equation}
Take $r=\lambda$ to see that 
\begin{equation}\label{scale2}
v^\lambda(t,x)=\lambda v^1(\lambda t, \sqrt{\lambda}x)
\end{equation}
and $r=1/t$ to obtain
\begin{equation}\label{scale3}
v^\lambda(t,x)=t^{-1}v^{\lambda t}(1,t^{-1/2}x).
\end{equation}
The following monotonicity properties are clear from (\ref{e_LapFunCanonlambda2}).
\begin{lemma} \label{lemma_monotonicity}
The map $x \to v^\lambda_t(x)$ is decreasing in $x$, and $\lambda \to v^\lambda_t(x)$ is increasing in $\lambda$.
\end{lemma}

We may let $\lambda\to\infty$ in \eqref{e_LapFunCanonlambda2} and \eqref{e_LapFunlambda2}
and use Dominated Convergence (for this in \eqref{e_LapFunCanonlambda2} use $1-e^{-\lambda X_t([x,\infty))}\le 1(X_t([x,\infty))>0)$
which is integrable with respect to $\N_0$) to see that $v^\lambda(t,x)\uparrow v^\infty(t,x)$ as $\lambda\uparrow\infty$, where
\begin{equation} \label{e_vinftyprob}
v^\infty_t(x) = \N_0\big( X_t([x,\infty)) > 0 \big)
 = -\log \big( P_{\delta_0}^X \left( X_t([x,\infty)) = 0 \right) \big).
\end{equation}
Note that
\begin{equation}\label{N0bound1}v_t^\infty(x)=\N_0(X_t([x,\infty))>0)\le \N_0(X_t(1)>0)=2/t,
\end{equation}
(see Theorem~II.7.2(iii) of \cite{P2002}) and in particular $v_t^\infty$ is finite for $t>0$. 
\begin{proposition} \label{prop_localunif}
$v^\lambda_t(x) \to v^\infty_t(x)$ uniformly on compact sets in $(0,\infty) \times \R$. In fact, there is uniform convergence for $(t,x) \in [a,\infty) \times [-R,R]$, for any $a>0$ and $R>0$.
\end{proposition}
\begin{proof}
Taking $\lambda \to \infty$ in \eqref{scale3}, we obtain that 
\begin{equation}\label{vinfscaling}v^\infty_t(x) = t^{-1} v^\infty_1(t^{-1/2}x).
\end{equation}
This fact and \eqref{scale3} imply that
\begin{equation}
v^\infty_t(x) - v^\lambda_t(x) = t^{-1} \left[ v^{\infty}_1(t^{-1/2} x) - v^{\lambda t}_1(t^{-1/2}x) \right]. \nonumber
\end{equation}
Let $0<a,R$. Then by the above, for $t\geq a$ we have
\begin{equation}
v^\infty_t(x) - v^\lambda_t(x) \leq a^{-1} \left[ v_1^{\infty}(t^{-1/2} x) - v_1^{\lambda a}(t^{-1/2}x) \right], \nonumber
\end{equation}
where we have used monotonicity in $\lambda$. Thus
\begin{align}
\hspace{ 5 mm}\sup_{t \ge a} \sup_{|x| \leq R} v^\infty_t(x) - v^\lambda_t(x) 
&\leq a^{-1} \sup_{t \ge a} \sup_{|x| \leq R} v_1^{\infty}(t^{-1/2} x) - v_1^{\lambda a}(t^{-1/2}x) \nonumber
\\ &\leq a^{-1} \sup_{|x| \leq a^{-1/2}R} v_1^{\infty}( x) - v_1^{\lambda a}(x). \nonumber
\end{align}
The continuity of $v^\infty_1$ (e.g., from \eqref{e_vinftyprob}) and Dini's Theorem imply that $v^\lambda_1 \uparrow v^\infty_1$ uniformly on compact sets, and the result follows.
\end{proof}

\begin{theorem}\label{vinfinitypde}
$v^\infty_t(x) \in C^{1,2}((0,\infty) \times \R)$ and is the unique non-negative solution to the PDE
\begin{align} \label{e_dualPDEinfty2}
(i)& \,\,\,\frac{\partial v_t}{\partial t} = \frac 1 2 \frac{\partial^2 v_t}{\partial x^2} - \frac{v^2_t}{2} \text{   on   } (t,x) \in (0,\infty) \times \R, \nonumber
\\ (ii)&\,\,\,  \lim_{t \downarrow 0} \int_U v^\infty_t (x)\,dx = + \infty \,\,\,\, \forall \, U\subseteq \R \text{ open such that } U \cap (-\infty,0] \neq \emptyset,
\\ & \,\,\, \lim_{t \downarrow 0} \int_K v^\infty_t (x)\,dx = 0 \,\,\,\, \forall \, K \subseteq  \R\text{ compact such that } K \subset(0,\infty). \nonumber
\end{align}
\end{theorem}
\begin{proof}
From Proposition \ref{prop_localunif} we have the local uniform convergence of $v^\lambda_t$ to $v^\infty_t$. The family $\{v^\lambda_t\}$ therefore satisfies the conditions of Theorem 3.10 of \cite{MV1999}, which shows that $v^\infty(t,x)$ solves \eqref{e_dualPDEinfty2}. Uniqueness follows by Theorem 3.5 of the same paper. 
\end{proof}

Recall that $G(x) = v^\infty_1(x)$.

\begin{lemma} \label{lemma_Gproperties}
(a) For all $t>0$, $v^\infty_t(x) = t^{-1} G(t^{-1/2} x)$ for all $x \in \R$.\\
(b) $G >0$ and is decreasing. \\
(c) $G \in C^\infty(\R)$ and is the unique positive $C^2$ solution to the ordinary differential equation
\begin{align} \label{ODE_G}
& \, \, \frac 1 2 G''(x) + \frac x 2 G'(x) + G(x) - \frac 1 2 G(x)^2 = 0 
\end{align}  
with boundary conditions $\lim_{x \to \infty} x^2G(x) = 0$, $\lim_{x \to -\infty}G(x) = 2$.\\
(d) There is a constant $c_{\ref{lemma_Gproperties}}$ such that:\\
(i)
\begin{align*}
G(x) &\leq c_{\ref{lemma_Gproperties}} |x|e^{-x^2/2} \,\, \forall \, x >2,
\\0\le 2 - G(x) &\leq c_{\ref{lemma_Gproperties}} |x| e^{-x^2/2} \,\, \forall \, x< -2.
\end{align*}
(ii) $G' \le 0$, and
\begin{align*}
|G'(x)| \leq c_{\ref{lemma_Gproperties}} x^2 e^{-x^2/2} \,\, \forall \, x > 2
\\ |G'(x)| \leq c_{\ref{lemma_Gproperties}} e^{-x^2/2}\,\, \forall \, x \le 0.
\end{align*}

\end{lemma}
\begin{proof} (a) is a restatement of \eqref{vinfscaling} and
 (b) is obvious from (\ref{e_vinftyprob}). \\
 
 (c) $G=v^\infty_1$ is $C^2$ by Theorem~\ref{vinfinitypde}.  Noting that $v^\infty_t(x) = t^{-1} G(t^{-1/2} x)$  solves \eqref{e_dualPDEinfty2}(i), we can
use the chain rule to see
\begin{align}
-\frac{1}{t^2} G(t^{-1/2}x) - \frac{1}{2t^2} (t^{-1/2}x) G'(t^{-1/2}x) &= \frac{1}{2 t^{2}} G''(t^{-1/2}x) - \frac{1}{2t^{2}} G(t^{-1/2}x)^2 \nonumber
\\-G(y) - \frac y 2 G'(y) &= \frac 1 2 G''(y) - \frac{1}{2} G(y)^2, \nonumber
\end{align}
This proves that $G$ solves (\ref{ODE_G}).  To see that it is $C^3$, we note that when we solve (\ref{ODE_G}) for $G''$, the expression is differentiable because $G$ and $G'$ are differentiable. Proceeding by induction we see that $G$ is $C^\infty$. The boundary conditions will clearly follow from (d) below. It remains to prove uniqueness. Let $H$ be any positive $C^2$ solution of \eqref{ODE_G} satisfying the given boundary conditions and set $u(t,x)=t^{-1}H(t^{-1/2}x)$. Then, reversing
the above steps one easily sees that $u$ is a $C^2$ solution of \eqref{e_dualPDEinfty2}(i). Let $0<a<b$ and choose $t>0$ small enough so that $y^2H(y)<\epsilon a$ for $y\ge at^{-1/2}$. Then 
\begin{align*}
\int_a^b v(t,x)dx=t^{-1/2}\int_a^b H(t^{-1/2}x)t^{-1/2}dx&\le t^{-1/2}\int_{at^{-1/2}}^\infty H(y)dy\\
&\le \epsilon a t^{-1/2}\int_{a t^{-1/2}}^\infty y^{-2}dy\\
&=\epsilon.
\end{align*}
This proves the second boundary condition in \eqref{e_dualPDEinfty2}(ii).  The first boundary condition is even easier to establish.  So by the uniqueness in Theorem~\ref{vinfinitypde}, $H(x)=u(1,x)=v^\infty(1,x)=G(x)$.\\

 (d)(i)
 To deduce the bound for positive $x$, we note that
\begin{align}
G(x) &= \N_0 \left( X_1([x,\infty))>0\right) \nonumber
\\ &\leq \N_0 \left(X_s([x,\infty))>0 \text{ for some } s \le 1 \right) \nonumber
\\ &\leq c_{\ref{th_DIP33}}|x| e^{-x^2/2}, \nonumber
\end{align}
for all $x > 2$, by Theorem \ref{th_DIP33}(i).  The lower bound on $2-G(x)$ is immediate from
\eqref{N0bound1} (for all $x$).  For $x<-2$, we have
\begin{align} \label{e_Glemma0}
2 - G(x) &= \N_0\left(X_1(1) > 0 \right) - \N_0\left(X_1([x,\infty)) >0 \right) \nonumber
\\ &\leq \N_0 \left( X_1((-\infty,x]) >0 \right)
\\ &\leq c_{\ref{th_DIP33}}|x|e^{-x^2/2}, \nonumber
\end{align}
again using Theorem \ref{th_DIP33}(i) and symmetry.\\

(ii) By (b) $G'\le 0$. Now note that (\ref{ODE_G}) can be rewritten as
\begin{equation} 
\left( e^{x^2/2} G'(x) \right) ' = e^{x^2/2} G(x)(G(x) - 2). \nonumber
\end{equation}
Integrating the above, for $x_0,x\in \R$ we get
\begin{align} \label{e_Glemma1}
G'(x) = e^{-x^2/2} \left[e^{x_0^2/2} G'(x_0) + \int_{x_0}^x e^{y^2/2} G(y) (G(y) - 2)\, dy \right]. 
\end{align}
For $x > x_0\ge 2$, both terms in the above are non-positive, and, if $c$ is the provisional constant arising in (i), we can use part (i) to deduce that
\begin{align} 
|G'(x)| &\leq e^{-x^2/2} \left[ e^{x_0^2/2} |G'(x_0)| + 2c \int_{x_0}^x |y| \, dy \right] \nonumber
\\ &\leq (c_1(x_0) + c_2 x^2) e^{-x^2/2} \leq c'x^2 e^{x^2/2} \nonumber
\end{align}
for $x > 2$. For $x \le 0= x_0$, we note that the integral in (\ref{e_Glemma1}) has its sign reversed, so is positive. Because $G'(x) \le 0$, $|G'(x)|$ is bounded above by the absolute value of the
 first term in \eqref{e_Glemma1}, which gives the required bound.
\end{proof}



Recall from Section~\ref{KOU} that if $G$ is as above, then  $A^G$ is the generator of a killed Ornstein-Uhlenbeck process with killing function $G$, $-\lambda_0^G$ is its lead eigenvalue, and $\psi^G_0$ denotes its corresponding unit eigenfunction in $\cL^2(m)$.
\begin{proposition} \label{prop_leadeig}
For some constant $c_{\ref{prop_leadeig}}>0$, $\psi^G_0(x) = -c_{\ref{prop_leadeig}}e^{x^2/2} G'(x)$ with eigenvalue $-\lambda^G_0 = -1$.
\end{proposition}
\begin{proof}
Recall that $G$ is the $C^\infty$ solution of (\ref{ODE_G}). Rearranging the equation, we can write $G''(x) = -xG'(x) - 2G(x) + G(x)^2$. $G$ is  $C^\infty$, so we can differentiate again to obtain a new ODE.
\begin{align} \label{e_ODE_Gprime}
\frac 1 2 G''' + \frac 1 2 G' + \frac 1 2 xG'' + G' - \frac 1 2 2G G' &= 0 \nonumber
\\ \iff \frac 1 2 G''' + \frac 1 2 xG'' + \frac 3 2 G'  - G G' &= 0
\end{align}
Let $\psi(x) = e^{x^2/2} G'(x)$. Let us first observe that $\psi \in \cL^2(m)$ because
\begin{align}
\int \psi(x)^2 \, dm(x) = (2\pi)^{-1/2} \int G'(x)^2 e^{x^2/2} \, dx< \infty, \nonumber
\end{align}
where the integral converges by Lemma \ref{lemma_Gproperties}(d)(ii). We compute the first and second derivatives of $\psi$:
\[ \psi'(x) = x e^{x^2/2} G'(x) + e^{x^2/2} G''(x)\]
and 
\begin{align*}
\psi''(x) &= x^2 e^{x^2/2} G'(x) + e^{x^2/2} G'(x) +  xe^{x^2/2} G''(x) + xe^{x^2/2}G''(x) + e^{x^2/2} G'''(x)    \nonumber
\\ &= x^2 e^{x^2/2} G'(x) + e^{x^2/2} G'(x) + 2xe^{x^2/2}G''(x) + e^{x^2/2} G'''(x).
\end{align*} 
Using the above, we evaluate $A^G \psi$.
\begin{align*}
A^G\psi &= \frac 1 2 \psi ''(x) - \frac 1 2 x\psi'(x) - \psi(x) G(x)
\\ &= e^{x^2/2} \left[\frac 1 2 x^2  G'(x) +\frac 1 2 G'(x) + xG''(x) + \frac 1 2G'''(x) \right] 
\\& \hspace{4 mm}- e^{x^2/2} \left[ \frac 1 2 x^2 G'(x) + \frac 1 2xG''(x)\right] - e^{x^2/2} \bigg[ G(x) G'(x) \bigg]
\\ &= e^{x^2/2}\left[ \frac 1 2 G'''(x) + \frac 1 2 xG''(x) + \frac 1 2 G'(x) - G(x) G'(x) \right]
\\ &= e^{x^2/2}\left[ \frac 1 2 G'''(x) + \frac 1 2 xG''(x) + \frac 3 2 G'(x) - G(x) G'(x) \right] - e^{x^2/2} G'(x)
\\ &= - \psi(x),
\end{align*}
where the last equality is due to (\ref{e_ODE_Gprime}). Moreover, $G'(x) \le0$ for all $x$, so $-e^{x^2/2}G'(x) \ge 0$, and we have already seen that it is in $\cL^2(m)$. Therefore $\psi$ is a non-positive eigenfunction of $A^G$ with eigenvalue $-1$.  Clearly $\psi$ cannot be orthogonal to the lead eigenfunction $\psi^G_0>0$ (recall Theorem~\ref{thm_OU}(a)). It follows that $\psi_0^G=-c_{\ref{prop_leadeig}}\psi$ for some normalizing constant $c_{\ref{prop_leadeig}}>0$ and hence the corresponding lead eigenvalue is $-1$.
\end{proof}

The next result gives a bound on the left tail of the distribution of $X_t([x,\infty))$ which will play an important
role in the proof of Proposition~\ref{prop_boundaryGrowthBd}.  We do not know what the ``correct" power law behaviour is, but
see the Remark at the end of this section for a possible answer. 

\begin{proposition}\label{new3.6}
For $0<p<1/6$ and $t>0$ there is a constant $C_{\ref{new3.6}}=C_{\ref{new3.6}}(p,t)$ such that 
\begin{equation*} P^X_{\delta_0}\Bigl(0<X_t([x,\infty))\le \frac{1}{\lambda}\Bigr)\le C_{\ref{new3.6}}\lambda^{-p}\quad \text{for all }x\in\R \text{ and }\lambda>0.\end{equation*}
\end{proposition}
\begin{proof} It clearly suffices to consider $\lambda\ge 1$, which is assumed until otherwise indicated. Let $0<p<1/6$ and $\veps=\veps(p)\in(0,1/6-p)$.  Assume we are working under a probability, $P$, for which $H$ is a historical process defining the super-Brownian motion $X$ starting at $\delta_0$ and let $\cF_t$ be the right-continuous completed filtration generated by $H$. $E$ will denote expectation with respect to $P$.
Recall $h(r)$, $\rho(c)$ and $\delta(t,c)$ are as in \eqref{Hmod} and \eqref{deltarate}. By \eqref{rhoasymp} we may choose $c=c(p)$ large enough
so that $\rho(c)\veps\ge p$ and so by \eqref{deltarate},
\begin{equation}\label{modcontbound}
P(\delta(t+1,c)\le \lambda^{-\veps})\le C_{\ref{deltarate}}(t+1)\lambda^{-\veps\rho(c)}\le C_{\ref{deltarate}}(t+1)\lambda^{-p}\text{ for all }\lambda\ge 1.
\end{equation}
By (II.5.11) and (II.5.12) of \cite{P2002},
\begin{equation}\label{1aa}
E^X_{X_0}(e^{-\lambda X_t(1)})=\exp\Bigl(\frac{-2\lambda X_0(1)}{2+\lambda t}\Bigr);\quad P_{X_0}^X(X_t(1)=0)=\exp(-2X_0(1)/t),
\end{equation}
and so for $\lambda\ge 1$,
\begin{align}\label{smallmassbound}
\nonumber P(0<X_t(1)\le \lambda^{-1/6})&\le eE(1(X_t(1)>0)\exp(-\lambda^{1/6}X_t(1)))\\
\nonumber &=e\Bigl[\exp\Bigl(\frac{-2\lambda^{1/6}}{2+\lambda^{1/6}t}\Bigr)-\exp(-2/t)\Bigr]\\
\nonumber &\le 2e\Bigl[\frac{1}{t}-\frac{\lambda^{1/6}}{2+\lambda^{1/6}t}\Bigr]\\
&\le 4et^{-2}\lambda^{-1/6}.
\end{align}
Let 
\[E=E_{x,\lambda}=\{0<X_t([x,\infty))\le 1/\lambda,\ \delta(t+1,c)>\lambda^{-\veps}, X_t(1)>\lambda^{-1/6}\}.\]
Then by \eqref{modcontbound} and \eqref{smallmassbound} it suffices to show
\begin{equation}\label{Ebound}
P(E_{x,\lambda})\le C(t,p)\lambda^{-p}\quad\text{ for all $x\in\R$ and $\lambda\ge 1$}.
\end{equation}

Assume for now that $x\in\R$ and $\lambda\ge 1$.  Note that if $\tau(\lambda)=\tau^{\lambda^{-2/3}}(t)$ (recall $\tau^\veps$ is as in \eqref{def_taueps}), then 
\begin{equation}\label{tauonE} \text{on }E_{x,\lambda}\text{ we have, } -\infty<\tau(\lambda)<x\text{ and }X_t([\tau(\lambda),\infty))=\lambda^{-2/3}.
\end{equation}
Introduce
\[E^1_{x,\lambda}=E\cap\{x-\tau(\lambda)\ge \lambda^{-1/6}\}\text{ and }E^2_{x,\lambda}=E\cap\{x-\tau(\lambda)< \lambda^{-1/6}\}.\]
We consider $E^1$ first.  Set 
\[\beta=\frac{1}{3}+\veps.\]
Then there is a $\underline\lambda=\underline\lambda(c,\veps,t)=\underline\lambda(p,t)\ge 1$ such that 
\begin{equation}\label{incbound}
2ch(\lambda^{-\beta})<\lambda^{-1/6}\text{ and }\lambda^{-\beta}<t/2\text{ for }\lambda\ge \underline\lambda.
\end{equation}
Until otherwise indicated we will assume now that $\lambda\ge\underline\lambda$. Define 
\[\zeta_{t-\lambda^{-\beta}}=\inf\{s\ge 0:H_{s+t-\lambda^{-\beta}}(\{y:y_{t-\lambda^{-\beta}}\ge x-ch(\lambda^{-\beta})\})=0\}.\]
It follows from \eqref{histcondlaw} that for $u=0$ or $\lambda^{-\beta}$,
\begin{align}\label{Fellerdiff}
&\text{conditional on $\cF_{t-u}$, $Z_s=H_{t-u+s}(\{y:y_{t-\lambda^{-\beta}}\ge x-ch(\lambda^{-\beta})\})$ $(s\ge 0)$ is equal in law to}\\
\nonumber&\text{the Feller diffusion $(X_s(1),s\ge 0)$ starting at $H_{t-u}(\{y:y_{t-\lambda^{-\beta}}\ge x-ch(\lambda^{-\beta})\})$.}
\end{align}
Throughout this proof we will assume the Feller diffusion $X_s(1)$ starts at $x_0\ge 0$ under $P_{x_0}$. 
On $E^1$ we have \break
$\lambda^{-\beta}\le \lambda^{-\veps}<\delta(t+1,c)$ and so by the modulus of continuity \eqref{Hmod},
\[H_t(\{y:y_{t-\lambda^{-\beta}}\ge x-ch(\lambda^{-\beta})\})\ge H_t(\{y:y_t\ge x\})=X_t([x,\infty))>0.\]
This implies that (use \eqref{Fellerdiff} with $u=\lambda^{-\beta}$ to see that $Z_s$ sticks at zero when it hits zero) 
\begin{equation}\label{zetaone}
\zeta_{t-\lambda^{-\beta}}>\lambda^{-\beta}\text{ on }E^1_{x,\lambda}.
\end{equation}
Now again use the modulus of continuity and then \eqref{incbound} to that on $E^1$, 
\begin{align*}
H_t(\{y:y_{t-\lambda^{-\beta}}\ge x-ch(\lambda^{-\beta})\})&\le H_t(\{y:y_t\ge x-2ch(\lambda^{-\beta})\})\\
&\le X_t([x-\lambda^{-1/6},\infty))\quad\text{(by \eqref{incbound})}\\
&\le X_t([\tau(\lambda),\infty))\qquad\text{(since $x-\tau(\lambda)\ge \lambda^{-1/6}$ on $E^1$)}\\
&=\lambda^{-2/3},
\end{align*}
the last by \eqref{tauonE}.  Use the above fact that $H_t(\{y:y_{t-\lambda^{-\beta}}\ge x-ch(\lambda^{-\beta})\})\le \lambda^{-2/3}$ on $E^1$ and condition on $\cF_t$ (recall
 \eqref{Fellerdiff} with $u=0$) to conclude that 
\begin{align}\label{E1bound1}
P(E^1\cap\{\zeta_{t-\lambda^{-\beta}}>\lambda^{-\beta}+\lambda^{-1/2}\})&\le E(1(H_t(\{y:y_{t-\lambda^{-\beta}}\ge x-ch(\lambda^{-\beta}))\le \lambda^{-2/3})P(Z_{t+\lambda^{-1/2}}>0|\cF_t))\\
\nonumber&\le P_{\lambda^{-2/3}}(X_{\lambda^{-1/2}}(1)>0)\\
\nonumber&=1-\exp\Bigl(-\frac{2\lambda^{-2/3}}{\lambda^{-1/2}}\Bigr)\qquad\text{(by \eqref{1aa})}\\
\nonumber&\le 2\lambda^{-1/6}.
\end{align}
So \eqref{zetaone} and \eqref{E1bound1} show that 
\begin{equation}\label{E1bound2}P(E^1\cap\{\zeta_{t-\lambda^{-\beta}}\notin[\lambda^{-\beta},\lambda^{-\beta}+\lambda^{-1/2}]\})\le 2\lambda^{-1/6}\text{ for }\lambda\ge \underline\lambda(p,t).
\end{equation}
Let 
\[M(\omega)=X_{t-\lambda^{-\beta}}([x-ch(\lambda^{-\beta}),\infty)).\]
If $\zeta=\inf\{s\ge 0:X_s(1)=0\}$ is the lifetime of the Feller diffusion $X_s(1)$, then we may apply \eqref{Fellerdiff} with $u=\lambda^{-\beta}$ to see that 
\begin{align}\label{E1bound3}
P(\zeta_{t-\lambda^{-\beta}}\in[\lambda^{-\beta},\lambda^{-\beta}+\lambda^{-1/2}])&=E(E_M(\zeta\in[\lambda^{-\beta},\lambda^{-\beta}+\lambda^{-1/2}]))\\
\nonumber&=E(P_M(X_{\lambda^{-\beta}+\lambda^{-1/2}}(1)=0)-P_M(X_{\lambda^{-\beta}}(1)=0))\\
\nonumber&=E\Bigl(\exp\Bigl(\frac{-2M}{\lambda^{-\beta}+\lambda^{-1/2}}\Bigr)-\exp\Bigl(\frac{-2M}{\lambda^{-\beta}}\Bigr)\Bigr)\ \ \quad\text{(by \eqref{1aa})}\\
\nonumber&\le E\Bigl(\exp(-2M/(\lambda^{-\beta}+\lambda^{-1/2}))2M\Bigr)[1/\lambda^{-\beta}-1/(\lambda^{-\beta}+\lambda^{-1/2})]\\
\nonumber&=E\Bigl(\exp(-2M/(\lambda^{-\beta}+\lambda^{-1/2}))2M/(\lambda^{-\beta}+\lambda^{-1/2})\Bigr)[\lambda^{-1/2}/\lambda^{-\beta}]\\
\nonumber&\le \lambda^{-((1/2)-\beta)}=\lambda^{-1/6+\veps}\le \lambda^{-p},
\end{align}
where we have used $\sup_{x\ge 0}xe^{-x}=e^{-1}\le 1$ in the last line. Combining \eqref{E1bound2} and \eqref{E1bound3} we arrive at 
\[P(E^1_{x,\lambda})\le 3\lambda^{-p}\quad\text{for all }\lambda\ge \underline\lambda(p,t),\ \  x\in\R.\]
This then implies that for some $c_{\ref{E1bound4}}=c_{\ref{E1bound4}}(p,t)$,
\begin{equation}\label{E1bound4}
P(E^1_{x,\lambda})\le c_{\ref{E1bound4}}(p,t)\lambda^{-p}\quad\text{for all }\lambda\ge 1, \ \ x\in\R.
\end{equation}

Consider next $E^2=E^2_{x,\lambda}$ where for now $\lambda\ge 1$ and of course $x\in\R$.  Recall that $U_s=\sup(S(X_s))$.  On $E^2$, we have $\lambda^{-5/6}\le \lambda^{-\veps}\le \delta(t+1,c)$ and so by the modulus of continuity \eqref{Hmod},
\begin{align*}
P(E^2\cap\{U_{t+\lambda^{-5/6}}\ge x+ch(\lambda^{-5/6})\})&\le P(E^2\cap\{H_{t+\lambda^{-5/6}}(\{y:y_t\ge x\})>0\})\\
&\le P_{\lambda^{-1}}(X_{\lambda^{-5/6}}(1)>0),
\end{align*}
where we have used \eqref{histcondlaw} with $\delta=0$, and $H_t(\{y:y_t\ge x\})\le 1/\lambda$ on $E$ in the last line.  Now use \eqref{1aa} to see that the above equals $1-\exp(-2\lambda^{-1}\lambda^{5/6})\le 2\lambda^{-1/6}$,
and so conclude that
\begin{equation}\label{E2bound1}
P(E^2\cap \{U_{t+\lambda^{-5/6}}\ge x+ch(\lambda^{-5/6})\})\le 2\lambda^{-1/6}\text{ for all }\lambda\ge 1, x\in\R.
\end{equation}
The modulus of continuity also implies
\begin{align*}
P&(E^2\cap\{U_{t+\lambda^{-5/6}}\le x-\lambda^{-1/6}-ch(\lambda^{-5/6})\})\\
&\le P(E^2\cap\{H_{t+\lambda^{-5/6}}(\{y:y_t\ge x-\lambda^{-1/6}\})=0\})\\
&\le P(E^2\cap\{H_{t+\lambda^{-5/6}}(\{y:y_t\ge \tau(\lambda)\})=0\})\quad\text{(recall $x-\tau(\lambda)<\lambda^{-1/6}$ on $E^2$)}\\
&=P_{\lambda^{-2/3}}(X_{\lambda^{-5/6}}(1)=0)\qquad\quad\text{(by\eqref{histcondlaw} with $\delta=0$, and \eqref{tauonE})}\\
&=\exp\Bigl(\frac{-2\lambda^{-2/3}}{\lambda^{-5/6}}\Bigr)=\exp(-2\lambda^{1/6})\le \lambda^{-1/6},
\end{align*}
the last since $\lambda\ge 1$.  The above inequality and \eqref{E2bound1} imply that 
\begin{equation}\label{E2bound2}
P(E^2\cap\{U_{t+\lambda^{-5/6}}\notin(x-\lambda^{-1/6}-ch(\lambda^{-5/6}),x+ch(\lambda^{-5/6}))\})\le 3\lambda^{-1/6}\ \ \forall \lambda\ge 1, x\in\R.
\end{equation}
Differentiate both sides of the scaling relationship in Lemma~\ref{lemma_Gproperties}(a) and so get
\begin{equation}\label{dvinftybnd}
\Bigl|\frac{\partial}{\partial x}v^\infty(t,x)\Bigr|\le t^{-3/2}\Vert G'\Vert_\infty.
\end{equation}
If $t'=t+\lambda^{-5/6}$, $x_1=x-\lambda^{-1/6}-ch(\lambda^{-5/6})$, $x_2=x+ch(\lambda^{-5/6})$, and $\lambda\ge \underline\lambda(p,t)$, then 
\begin{align*}P(U_{t'}\in(x_1,x_2])&=P(X_{t'}([x_2,\infty))=0)-P(X_{t'}([x_1,\infty))=0)\\
&=e^{-v^\infty_{t'}(x_2)}-e^{-v^\infty_{t'}(x_1)}\quad\text{(by \eqref{e_vinftyprob})}\\
&\le t^{-3/2}\Vert G'\Vert_\infty(x_2-x_1)\qquad\text{(by \eqref{dvinftybnd})}\\
&= t^{-3/2}\Vert G'\Vert_\infty(\lambda^{-1/6}+2ch(\lambda^{-5/6}))\\
&\le 2t^{-3/2}\Vert G'\Vert_\infty\lambda^{-1/6},
\end{align*}
where in the last line we used \eqref{incbound}. 
The above, together with \eqref{E2bound2}, implies that 
\[P(E^2_{x,\lambda})\le (2t^{-3/2}\Vert G'\Vert_\infty+3)\lambda^{-1/6}\quad\text{ for all } x\in\R, \lambda\ge\underline\lambda(p,t).\]
This in turn shows that for some $c_{\ref{E2bound3}}(p,t)$,
\begin{equation}\label{E2bound3}
P(E^2_{x,\lambda})\le c_{\ref{E2bound3}}(p,t)\lambda^{-1/6}\text{ for all }x\in\R, \lambda\ge 1.
\end{equation}
Combining \eqref{E1bound4} and \eqref{E2bound3}, we derive \eqref{Ebound}, as required. 
\end{proof}

An easy consequence of the above is a rate of convergence of $v^\lambda$ to $v^\infty$ as $\lambda\to\infty$.  This will play an important role in the proof of Theorem~\ref{thm_01-law} given
in the next section.

\begin{proposition}\label{prop_vlambdaROC}
For any $0<p<1/6$ there is a $C_{\ref{prop_vlambdaROC}}(p)$ such that 
\[\sup_x|v^\infty_t(x)-v^\lambda_t(x)|\le C_{\ref{prop_vlambdaROC}}(p)t^{-p-1}\lambda^{-p}\quad\text{for all }\lambda,t>0.\]
\end{proposition}
\begin{proof} Let $0<p<1/6$.  By \eqref{e_LapFunlambda2} and \eqref{e_vinftyprob},
\begin{align}\label{explb}
e^{-v_t^\lambda(x)}-e^{-v_\infty^\lambda(x)}&=E^X_{\delta_0}(e^{-\lambda X_t([x,\infty))}1(X_t([x,\infty))>0))\\
\nonumber&=E^X_{\delta_0}\Bigl(\int_0^\infty1(0<X_t([x,\infty))\le u)e^{-\lambda u}\lambda du\Bigr)\\
\nonumber&\le C_{\ref{new3.6}}(p,t)\int_0^\infty u^pe^{-\lambda u}\lambda du\quad\text{(Proposition~\ref{new3.6})}\\
\nonumber&=\Gamma(p+1)C_{\ref{new3.6}}(p,t) \lambda^{-p}.
\end{align}
Recalling from \eqref{N0bound1} that $v^\infty_t(x)\le 2/t$, we also have 
\begin{align}\label{expub}
e^{-v_t^\lambda(x)}-e^{-v_\infty^\lambda(x)}&\ge e^{-v^\infty_t(x)}(v^\infty_t(x)-v^\lambda_t(x))\ge e^{-2/t}(v^\infty_t(x)-v^\lambda_t(x)).
\end{align}
Combine \eqref{explb} and \eqref{expub} and set $t=1$ to see that 
\[\sup_x|v^\infty_1(x)-v^\lambda_1(x)|\le e^2\Gamma(p+1)C_{\ref{new3.6}}(p,1) \lambda^{-p}.\]
The required relation is now immediate from the scaling relations \eqref{scale3} and Lemma~\ref{lemma_Gproperties}(a). 
\end{proof}
\begin{remark} We do not believe $p=1/6$ is sharp in any way. 
Theorem~1.5 of \cite{MMP17} studies solutions of 
\[\frac{\partial u}{\partial t}=\frac{1}{2}\frac{\partial^2u}{\partial x^2}-\frac{u^2}{2},\qquad u_0=\lambda\delta_0.\]
In particular this paper shows (via a Feynman-Kac argument) that for some $0<\underline C(K)\le \overline C<\infty$,
\[\underline C(K)t^{-(1/2)-\lambda_0}\lambda^{-(2\lambda_0-1)}\le u^\infty_t(x)-u^\lambda_t(x)\le \overline Ct^{-(1/2)-\lambda_0}\lambda^{-(2\lambda_0-1)},\]
where $\lambda_0=\lambda_0^F$ (as in Theorem~\ref{dimwp1}) and the lower bound is valid for $\lambda\ge 
t^{-1/2}$ and $|x|\le K\sqrt t$. So naively changing $u^\lambda$ to $v^\lambda$ leads to replacing $F=u^\infty_1$ with $G=v^\infty_1$, and one might think that (the $t$ dependence is by scaling \eqref{scale3})
\begin{equation}\label{vrateconj}
v_t^\infty(x)-v^\lambda_t(x)\approx Ct^{-2\lambda_0^G}\lambda^{-(2\lambda_0^G-1)}=Ct^{-2}\lambda^{-1}\ \text{ as }\lambda\to\infty,
\end{equation}
where $\approx$ means bounded below and above for perhaps differing positive constants $C$. 
This rate does hold if $x=-\infty$, where (by \eqref{1aa} and \eqref{e_vinftyprob})
\[v_t^\infty(-\infty)-v^\lambda_t(-\infty)=\frac{2}{t}-\frac{2\lambda}{2+\lambda t}\sim 4t^{-2}\lambda^{-1}\ \text{ as }\lambda\to\infty.\]
However the proof in \cite{MMP17} relies on the scaling of $u^\lambda$, which differs from that of $v^\lambda$. Moreover there is some evidence that the convergence when $x\gg0$ is slower.  In fact a heuristic
argument suggests that the correct rate at $+\infty$ is given by $p=G(0)-1\in(0,1)$.  The last upper bound is obvious because $G(0) < G(-\infty)=2$. For the lower bound on $G(0)$, note that by \eqref{e_vinftyprob} we have $G(0) = \N_0 ( \{X_1([0,\infty)) > 0\})$, so by symmetry, 
\[2G(0) - \N_0(\{X_1([0,\infty))>0 \} \cap \{X_1((-\infty,0]) > 0 \}) = \N_0(\{X_1(1) > 0\}) = 2,\] 
thus implying that $G(0)>1$.
\end{remark}

\section{Proof of Theorem \ref{thm_01-law}}\label{secthm01}
We first establish a lower bound on the probability that $L_t$ has positive mass at distances of order $\sqrt t$ away from zero under canonical measure. This follows readily from moment calculations in \cite{H2018}.
\begin{lemma} \label{lemma_Ltposhalfline}
There is a finite constant $C_{\ref{lemma_Ltposhalfline}}$ and for all $k\ge0$, positive constants $ c_{\ref{lemma_Ltposhalfline}}(k)$,  such that for all $t > 0$ and $k\ge 0$, 
\begin{equation}\label{Lsquare} \N_0 ( L_t([k \sqrt t, \infty))^2 ) \leq \N_0 ( L_t(\R)^2 )\le C_{\ref{lemma_Ltposhalfline}}\, t^{1-2\lambda_0},\end{equation}
and
\begin{equation}\label{Lposlb}\N_0 ( \{ L_t([k\sqrt t,\infty)) >0 \} \big| X_t > 0 ) > c_{\ref{lemma_Ltposhalfline}}(k).\end{equation}
\end{lemma} 
\begin{proof} The first claim is immediate from Theorem A(d). 
 The second claim is an easy application of the second moment method as we now show. By Theorem~A(b) the first moment of $L_t([k \sqrt t,\infty))$  is
\begin{align}
\N_0 \big( L_t([k\sqrt t, \infty)) \big) &= C_A \,t^{-\lambda_0} \int 1(\sqrt t z \geq k \sqrt t) \psi_0(z) \, dm(z) \nonumber 
= C_A \,t^{-\lambda_0} \int_k^\infty \psi_0 \, dm = c(k)\, t^{-\lambda_0},
\end{align}
where $c(k)>0$.
Thus by the second moment method, we have
\[ \N_0 \big( \{L_t(k \sqrt t ,\infty) > 0 \}\big) \geq \frac{\N_0 \big( L_t(k\sqrt t, \infty) \big)^2}{\N_0 ( L_t(k \sqrt t, \infty)^2)} \geq \frac{(c(k) \, t^{-\lambda_0} )^2}{C_{\ref{lemma_Ltposhalfline}} \, t^{1-2\lambda_0}} =: 2c_{\ref{lemma_Ltposhalfline}}(k) \, t^{-1}. \]
Because $L_t = 0$ when $X_t = 0$ and $\N_0( \{X_t > 0\}) = 2/t$, this implies that $\N_0\big( \{L_t(k\sqrt t,\infty) > 0\} \, \big| \, X_t > 0 \big) \geq c_{\ref{lemma_Ltposhalfline}}(k)$.
\end{proof}
We begin the study of $X_t$ near the upper edge of its support.  Recall the notation $\tau^\epsilon(t)$ from \eqref{def_taueps} in the Introduction. 
We first obtain a preliminary upper bound for the mass of $X_t$ near $\tau^\epsilon(t)$.
\begin{lemma} \label{lemma_boundaryGrowthFK}
Let $t,\epsilon>0$ and $u> 0$. Then
\begin{align}
\N_0 \left( \int_{\tau^\epsilon(t) - u}^\infty X_t(x) dx \right) \leq eE_0^B\Bigl(\exp\Bigl\{-\int_0^t v_s^{\eps^{-1}}(B_s+u)ds\Bigr\}\Bigr),
 \nonumber
\end{align}
where $B$ is a standard Brownian motion under $P_0^B$ and $v^{\eps^{-1}}$ is as in \eqref{e_dualPDElambda2}.
If $X_0\in \cM_F(\R)\setminus\{0\}$, then\\
 $E^X_{X_0}\left(\int_{\tau^\epsilon(t)-u}^\infty X_t(x)ds\right)/X_0(1)$ is bounded by the same expression.
\end{lemma}
\begin{proof} As $t$ is fixed we will write $\tau^\epsilon$ for $\tau^\epsilon(t)$.
We begin by examining $\N_0 \left( \int_{\tau^\epsilon - u}^\infty X_t(x) dx \right)$. It is equal to
\begin{align}
\N_0 \left( \int 1(x+u > \tau^\epsilon) X_t(x)\, dx \right)  
&= \N_0 \left( \int 1\left(X_t([x+u,\infty)) < \epsilon \right) X_t(x) \,dx \right) \nonumber
\\ &\leq e \N_0 \left( \int e^{-\epsilon^{-1} X_t([x+u,\infty))} X_t(x)\, dx \right). \nonumber
\end{align}
(We note that the above is true both when $\tau^\epsilon \in \R$ and when $\tau^\epsilon = -\infty$, in which case $X_t(x+ u,\infty) \leq \epsilon$ for all $x \in \R$.) By Theorem 4.1.3 of Dawson and Perkins \cite{DP} and 
translation invariance, the above is equal to
\begin{align} \label{e_endlem1}
e E_0^B \left( \exp \left(-\int_0^t  u_{t-s}^{\epsilon^{-1}}(W_s-W_t-u) \, ds \right)\right),
\end{align}
where $W$ is a standard Brownian motion under $P^B_0$ and $u_s^\lambda(x)$ solves \eqref{e_dualPDElambda2} but with $u_0^\lambda=\lambda 1_{[0,\infty)}$. Clearly $u^\lambda_t(x)=v^\lambda_t(-x)$ ($v^\lambda_t$ as in \eqref{e_dualPDElambda2}).  So if $B_s=-W_{t-s}+W_t$, a new standard Brownian motion under $P_0^B$, the above equals
\begin{equation}\label{vepsbound}
eE_0^B\left(\exp\Bigl\{-\int_0^tu^{\epsilon^{-1}}_{t-s}(-B_{t-s}-u)ds\Bigr\}\right)=eE_0^B\left(\exp\Bigl\{-\int_0^tv_s^{\epsilon^{-1}}(B_s+u)ds\Bigr\}\right).
\end{equation}

Consider next $P^X_{X_0}$, for a non-zero initial condition $X_0$.  Then, as above, $E^X_{X_0}\left(\int_{\tau^\eps+u}^\infty X_t(x)dx\right)$ equals
\[E_{X_0}^X\left(\int 1(X_t([x+u,\infty))<\eps)X(t,x)dx\right),\]
which by Theorem 4.1.1 of \cite{DP} is bounded by 
\[\iint 1(X_t([x+u,\infty))<\eps)d\N_{x_0}dX_0(x_0)\le X_0(1)eE_0^B\Bigl(\exp\Bigl\{-\int_0^t v_s^{\eps^{-1}}(B_s+u)ds\Bigr\}\Bigr).\]
To obtain the left-hand side of the above, we have ignored the contribution to $X_t$ from particles unrelated to the individual 
selected at $x$ by $X_t$ (the quoted theorem in \cite{DP} giving the rigorous justification), and the inequality follows from the bound \eqref{vepsbound} and the fact that the above calculation applies, where now $W_0=x_0$, because $B$ remains a Brownian motion starting at $0$. 
\end{proof}

We can now give the proof of Proposition~\ref{prop_boundaryGrowthBd} (restated below for convenience).  The quantity of interest is bounded in terms of the survival probability of an Ornstein-Uhlenbeck process $Y$ killed at rate $G(Y_s)$, for which we know the lead eigenvalue is $-1$ by Proposition \ref{prop_leadeig}. This leads to the $u^2$ term in upper bound. Proposition \ref{prop_vlambdaROC} allows us to make the approximations which lead to the eigenvalue problem.\\

{\bf Proposition~\ref{prop_boundaryGrowthBd}. 
\it There is a non-increasing function, $c_{\ref{prop_boundaryGrowthBd}}(t)$, such that for all $t,\epsilon>0$ and $u> 0$:}\hfil\break
(a) For any $X_0\in \cM_F(\R)$, $E^X_{X_0}\left(\int_{\tau^\eps(t)-u}^\infty X_t(x)dx\right)\le c_{\ref{prop_boundaryGrowthBd}}(t)X_0(1)(u^2\vee\eps)$.

(b) 
$\N_0\left( \int_{\tau^\epsilon(t) - u}^\infty X_t(dx) \right) 
\leq c_{\ref{prop_boundaryGrowthBd}}(t)(u^2 \vee \epsilon)$.\\

\noindent{\it Proof.}
The results are trivial if $u>1$ so we may assume $u\le 1$. Suppose first that $1\ge u^2 \geq \epsilon$.  By Lemma~\ref{lemma_boundaryGrowthFK} it suffices to show
\begin{equation}\label{expbound}
E_0^B\Bigl(\exp\Bigl\{-\int_0^t v_s^{\eps^{-1}}(B_s+u)ds\Bigr\}\Bigr)
\le c_{\ref{prop_boundaryGrowthBd}}(t)(u^2 \vee \epsilon).
\end{equation}

By the scaling relation \eqref{scale3} the left-hand side of the above
equals
\begin{equation}\label{expbound2} E_0^{{B}} \left( \exp \left(-\int_0^t \frac 1 s  v_{1}^{\epsilon^{-1}s}\left(\frac{{B}_{s}}{\sqrt s} + \frac{u}{\sqrt{s}} \right) \, ds \right)\right).
\end{equation}
 We define $\hat{Y}_s = e^{-s/2}{B}_{e^s}$, which defines a stationary Ornstein-Uhlenbeck process on $\R$. As this process is reversible with respect to its stationary measure $m$, $Y_s = \hat{Y}_{-s} = e^{s/2} {B}_{e^{-s}}$ is also a stationary Ornstein-Uhlenbeck process. We denote its expectation by $E^Y$. An exponential time change ($s=e^{-\hat s}$)  shows that \eqref{expbound2} is equal to
\begin{align}
 E^Y \left( \exp \left(- \int_{-\log t}^{\infty} v_{1}^{\epsilon^{-1}e^{-\hat s} } \left( Y_{\hat s}+ u e^{\hat s/2} \right) \right) \, d\hat s \right) &=E^Y_m \left( \exp \left(- \int_{0}^\infty v_{1}^{\epsilon^{-1}te^{-s'} } \left( Y_{s'} + u t^{-1/2}e^{s'/2} \right) ds' \right)   \right). \nonumber
\end{align}
The equality follows from changing variables to $s' = \hat s + \log t$ and the stationarity of $Y$.
We next truncate the integral and then add and subtract a $v_1^\infty$ term. This shows that if $p\in(0,1/6)$, then \eqref{expbound2} is at most
\begin{align}
E^Y_m \Bigg( & \exp \Bigg( -\int_{0}^{\log{(1/u^2)}} v_{1}^{\epsilon^{-1}te^{-s} } \Bigg( Y_s + u t^{-1/2} e^{s/2} \Bigg) \Bigg)  ds \Bigg) \nonumber
\\ &= E^Y_m \Bigg( \exp \left(\int_{0}^{\log(1/u^2)} \left(v_{1}^{\infty} - v_{1}^{\epsilon^{-1}te^{-s} }\right) \left( Y_s + u t^{-1/2} e^{s/2} \right) \, ds\right)  \exp \left(-\int_{0}^{\log(1/u^2)} v_{1}^{\infty} \left( Y_s + u t^{-1/2} e^{s/2} \right) ds\right)  \Bigg) \nonumber
\\ &\leq  E^Y_m \Bigg(  \exp \left(C_{\ref{prop_vlambdaROC}}(p) \int_{0}^{\log(1/u^2)} \left(\epsilon^{-1} t e^{-s} \right)^{-p} ds\right)  \exp \left(-\int_{0}^{\log(1/u^2)} v_{1}^{\infty} \left( Y_s + u t^{-1/2} e^{s/2} \right) ds\right)  \Bigg). \nonumber
\end{align}
The inequality follows by Proposition \ref{prop_vlambdaROC}. Moreover, since $u^2 \geq \epsilon$,
\begin{align}
\int_{0}^{\log(1/u^2)} \left(\epsilon^{-1} t e^{-s} \right)^{-p} ds = t^{-p} \, \epsilon^p \big[ e^{ps}/p \big]_0^{\log(1/u^2)} \leq\frac{ t^{-p}}{p}\, \left(\frac{\epsilon}{u^2}\right)^p \leq \frac{t^{-p}}{p}.\nonumber
\end{align}
This bounds \eqref{expbound2} above by
\begin{align} \label{e_endmassprop1}
e^{C_{\ref{prop_vlambdaROC}}t^{-p}/p} E^Y_m \Bigg( \exp \left(-\int_{0}^{\log(1/u^2)} G \left( Y_s + u t^{-1/2} e^{s/2} \right) ds\right)  \Bigg),
\end{align}
where $C_{\ref{prop_vlambdaROC}}=C_{\ref{prop_vlambdaROC}}(p)$ and we recall $G = v^\infty_1$. 
Define $\Delta(s)$ by
\begin{equation}
\Delta(s) = \big| G(Y_s) - G \left( Y_s + u t^{-1/2} e^{s/2} \right)\big|. \nonumber
\end{equation}
$G'$ is continuous and has limit $0$ at $\pm \infty$, thus $\|G'\|_\infty < \infty$. By the Mean Value Theorem,
\[ \Delta(s) \leq \|G'\|_\infty \, t^{-1/2}\, ue^{s/2} .\]
Thus (\ref{e_endmassprop1}) is bounded above by
\begin{align} \label{e_endmassprop2}
& e^{ C_{\ref{prop_vlambdaROC}}t^{-p}/p} E^Y_m \Bigg( \exp \left(\int_0^{\log(1/u^2)} \Delta(s) \, ds \right) \exp \left(-\int_{0}^{\log(1/u^2)} G (Y_s)\, ds\right)  \Bigg) \nonumber
 \\ &\leq e^{C_{\ref{prop_vlambdaROC}}t^{-p}/p}E^Y_m \Bigg( \exp \left(\|G'\|_\infty t^{-1/2}
 \, u\int_0^{\log(1/u^2)}  e^{s/2} \, ds \right) \exp \left(-\int_{0}^{\log(1/u^2)} G (Y_s)\, ds\right)  \Bigg) \nonumber
 \\&\leq e^{ C_{\ref{prop_vlambdaROC}}t^{-p}/p+2\|G'\|_\infty t^{-1/2}}\,E^Y_m \Bigg( \exp \left(-\int_{0}^{\log(1/u^2)} G (Y_s)\, ds\right)  \Bigg). 
\end{align}
Let $c(t) = e^{C_{\ref{prop_vlambdaROC}}t^{-p}/p+2\|G'\|_\infty t^{-1/2}}$. The remaining term is the probability that an Ornstein-Uhlenbeck process killed at rate $G(Y_s)$ survives until time $\log(1/u^2)$. If $\rho^G$ is the lifetime of this process, we have bounded \eqref{expbound2} by
\begin{align}
c(t)E^Y_m \left( \exp \left(-\int_{0}^{\log(1/u^2)} G (Y_s)\, ds\right)  \right) &= c(t)P^Y_m \left( \rho^G > \log(1/u^2) \right) \,\nonumber
\\ &\leq c(t)\,Ce^{-\lambda_0^G (\log(1/u^2))} = c_{\ref{prop_boundaryGrowthBd}}(t)u^2. \nonumber
\end{align}
The  inequality follows from \eqref{OUSbound} in Theorem \ref{thm_OU}(b) and the final equality is by Proposition \ref{prop_leadeig}
and setting $c_{\ref{prop_boundaryGrowthBd}}(t) = Cc(t)$. This completes the proof when $u^2 \geq \epsilon$. If $u^2 <\epsilon$, we have for (b), say,
\begin{align} \nonumber
\N_0\left( \int_{\tau^\epsilon - u}^\infty X_t(dx) \right) &\leq \N_0\left( \int_{\tau^\epsilon - \sqrt \epsilon}^\infty X_t(dx) \right) \leq c_{\ref{prop_boundaryGrowthBd}}(t) \epsilon,
\end{align}
where the final inequality follows by applying the $(u')^2 \geq \epsilon$ case with $u' = \sqrt \epsilon$.
The argument for (a) is the same.
\qed\\

We are now ready to give the proof of Theorem~\ref{thm_01-law}.  As suggested in the Introduction, the method of proof is to decompose the measure $X_{t-\delta}$ into two measures, to the right and left of $\tau^\delta(t-\delta)$. We then show that there is a uniformly positive probability that, the measure to the right of $\tau^\delta(t-\delta)$ produces positive mass (at time $t$)  in $L_t$ on a set far enough to the right that the mass from the measure to the left of $\tau^\delta(t-\delta)$ does not interfere with it.\\

\noindent \emph{Proof of Theorem \ref{thm_01-law}.} First, consider $P^X_{X_0}$. Let $(\cF_t)$ denote the usual completed right-continuous filtration generated by the associated historical process, $H$.
Fix $t>0$. Let $\delta_n = 2^{-n}$ and only consider $n$ so that $\delta_n<t/2$. We will show that the martingale $P^X_{X_0} \big(L_t>0\, \big|\, \cF_{t-\delta_n} \big)$ is bounded below by a positive number a.s. on $\{X_t>0\}$, and so, as it converges to $1_{\{L_t > 0\}}$ a.s., the latter must be  $1$ a.s. on $\{X_t>0\}$.\\

Set $\tau_n = \tau^{\delta_n}(t-\delta_n)$, that is,
\[ \tau_n = \inf \{ x \in \R : X_{t-\delta_n}([x,\infty)) < \delta_n \}\ge -\infty.\]
Now invoke the decomposition in \eqref{histdecomp} and \eqref{histcondlaw} with $\tau=\tau_n$ and $\delta=\delta_n$.  That is, we define  random measures by 
\begin{equation} \label{e_measuredecomp}
X_{t-\delta_n}^R(dx) = 1_{\{x\ge \tau_n\}}\, X_{t-\delta_n}(dx), \hspace{ 8 mm} X_{t-\delta_n}^L(dx) = 1_{\{x<\tau_n\}}\, X_{t-\delta_n}(dx),
\end{equation}
and define measure-valued processes by 
\begin{align*}
&\hat X^R_s(\phi)=\hat X^{R,\tau_n,\delta_n}_s(\phi)=\int\phi(y_{t-\delta_n+s})1(y_{t-\delta_n}\ge\tau_n)H_{t-\delta_n+s}(dy),\\
&\hat X^L_s(\phi)=\hat X^{L,\tau_n,\delta_n}_s(\phi)=\int\phi(y_{t-\delta_n+s})1(y_{t-\delta_n}<\tau_n)H_{t-\delta_n+s}(dy)\\
&\hat X_s(\phi)=\hat X^{R}_s(\phi)+\hat X^L_s(\phi)\ (=X_{t-\delta_n+s}(\phi)).
\end{align*}
Therefore by \eqref{histdecomp} and \eqref{histcondlaw}, 
 \begin{align}\label{Xtdecomp}
&X_t=\hat X_{\delta_n}=\hat{X}_{\delta_n}^R+\hat{X}_{\delta_n}^L,\text{ where conditional on $\cF_{t-\delta_n}$, $\hat{X}^R$ and $\hat{X}^L$ are independent super-Brownian motions}\\
\nonumber&\text{ with initial states $X^R_{t-\delta_n}$ and $X^L_{t-\delta_n}$, respectively. }
\end{align}
Below we will argue conditionally on $\cF_{t-\delta_n}$ and hence work with this pair of independent super-Brownian motions, $\hat X^R$ and $\hat X^L$, with initial laws $X^R_{t-\delta_n}$ and $X^L_{t-\delta_n}$.  
We apply the cluster decomposition (\ref{e_clusterdecomp1})  to each of these super-Brownian motions
to conclude
\begin{equation} \label{e_clusterdecomp3}
\hat{X}^R_{\delta_n} \stackrel{\cD}{=} \sum_{i =1}^{N_R} \hat{X}^{R,i}_{\delta_n} ,\quad \hat{X}^L_{\delta_n} \stackrel{\cD}{=} \sum_{i =1}^{N_L} \hat{X}^{L,i}_{\delta_n},
\end{equation}
where $N_R$ is Poisson with rate $2X^R_{t-\delta_n}(1)/\delta_n$, $N_L$ is an independent Poisson r.v. with rate $2X^L_{t-\delta_n}(1)/\delta_n$, and, conditional on $(N_R,N_L)$, $\{\hat{X}^{R,i}_{\delta_n}:i\le N_R\}$ are iid with law $\int\N_x(X_{\delta_n}\in\cdot|X_{\delta_n}>0)X^R_{t-\delta_n}(dx)/X^R_{t-\delta_n}(1)$ and  
$\{\hat{X}^{L,i}_{\delta_n}:i\le N_L\}$ are iid with law $\int\N_x(X_{\delta_n}\in\cdot|X_{\delta_n}>0)X^L_{t-\delta_n}(dx)/X^L_{t-\delta_n}(1)$. These last two collections are also conditionally independent. (Note also that if $X^L_{t-\delta_n}=0$, say, then $N_L=0$ and so there are no clusters to describe.) 
Let $\hat{L}^R_{\delta_n}$, $\hat{L}^L_{\delta_n}$ and $\hat{L}_{\delta_n}$denote the boundary local times of 
$\hat{X}^R_{\delta_n}$, $\hat{X}^L_{\delta_n}$ and $\hat{X}_{\delta_n}$, respectively. By \eqref{Xtdecomp} and applying Theorem~\ref{THMW} conditionally on $\cF_{t-\delta_n}$ we have
\begin{align} \label{e_localtimesplit}
\hat{L}_{\delta_n}(dx) &= 1_{\{\hat{X}^L_{\delta_n}(x) = 0\}}\hat{L}^R_{\delta_n}(dx) +1_{\{\hat{X}^R_{\delta_n}(x) = 0\}}\hat{L}^L_{\delta_n}(dx).
\end{align}
We also let $\hat{L}^{R,i}_{\delta_n}$ denote the boundary local time of $\hat{X}^{R,i}_{\delta_n}$. 
If  $\mu \otimes \nu$ denotes product measure, it follows that
\begin{align}
\nonumber P^X_{X_0}&(L_t([\tau_n+3\sqrt{\delta_n},\infty))>0|\cF_{t-\delta_n})\\
\nonumber&= P^X_{X_0}(\hat{L}_{\delta_n}([\tau_n+3\sqrt{\delta_n},\infty))>0|\cF_{t-\delta_n})
\\
\nonumber&\ge P^X_{X^R_{t-\delta_n}}\otimes P^X_{X^L_{t-\delta_n}}\Bigl(\int 1(x\ge \tau_n+3\sqrt{\delta_n})
1(\hat{X}^L_{\delta_n}(x)=0)\hat{L}^R_{\delta_n}(dx)>0\Bigr)\quad\text{(by \eqref{Xtdecomp} and \eqref{e_localtimesplit})}\\
\label{Lposbnd}&\ge P^X_{X^R_{t-\delta_n}}(\hat{L}^R_{\delta_n}([\tau_n+3\sqrt{\delta_n},\infty))>0)P^X_{X^L_{t-\delta_n}}(\hat{X}^L_{\delta_n}([\tau_n+3\sqrt{\delta_n},\infty))=0).
\end{align}
Now work on $\{\tau_n>-\infty\}\in\cF_{t-\delta_n}$ and consider the first term in \eqref{Lposbnd}. 
In this case
$X^R_{t-\delta_n}(1)=\delta_n$, and so $N_R$ is Poisson with mean $2$.  Therefore by restricting
to $\{N_R=1\}$ and noting that in this case $\hat L^R_{\delta_n}=\hat{L}^{R,1}_{\delta_n}$, we have 
\begin{align*}
P^X_{X^R_{t-\delta_n}}(\hat{L}^R_{\delta_n}([\tau_n+3\sqrt{\delta_n},\infty))>0)
&\ge 2e^{-2}P^X_{X^R_{t-\delta_n}}(\hat{L}^{R}_{\delta_n}([\tau_n+3\sqrt{\delta_n},\infty))>0|N_R=1)\\
&=2e^{-2}\int_{\tau_n}^\infty \N_{x}(L_{\delta_n}([\tau_n+3\sqrt{\delta_n},\infty))>0|X_{\delta_n}>0)X^R_{t-\delta_n}(dx)/\delta_n\\
&=2e^{-2}\int_{\tau_n}^\infty \N_0(L_{\delta_n}([\tau_n-x+3\sqrt{\delta_n},\infty))>0|X_{\delta_n}>0)X^R_{t-\delta_n}(dx)/\delta_n\\
&\ge2e^{-2}\N_0(L_{\delta_n}([3\sqrt{\delta_n},\infty))>0|X_{\delta_n}>0),
\end{align*}
where the last line again uses $X^R_{t-\delta_n}(1)=\delta_n$ on $\{\tau_n>-\infty\}$.  Therefore Lemma~\ref{lemma_Ltposhalfline} and \eqref{Lposbnd} now imply that on $\{\tau_n>-\infty\}$,
\begin{equation} \label{e_condexpLt4}
P^X_{X_0} \big(1_{\{ L_t([\tau_n + 3\sqrt{\delta_n},\infty)) > 0\}} \big| \, \cF_{t-\delta_n} \big) \geq 2e^{-2}c_{\ref{lemma_Ltposhalfline}}(3)\,\times P^X_{X^L_{t-\delta_n}} \big( \hat{X}^L_{\delta_n}([ \tau_n + 3 \sqrt{\delta_n}, \infty)) = 0\big).
\end{equation}

It remains to handle the final probability. 
We will consider events on which it has a uniform lower bound and which will occur infinitely often in $n$. For $K \in \N$, define an event $A_{K,n} \in \cF_{t-\delta_n}$ by
\begin{equation} \label{e_AKn_def}
A_{K,n} = \left\{\int_3^\infty w e^{-w^2/2}\,X_{t-\delta_n}( \tau_n -(w-3) \sqrt{\delta_n}) \, dw \leq K \sqrt{\delta_n} \right\}.
\end{equation}
(Note that the $A_{K,n}$ depends only on mass to the left of $\tau_n$, and so the measure in the integral is equal to $X^L_{t-\delta_n}$.) Noting that $X_{t-\delta_n}(-\infty) = 0$, we see that $\{\tau_n=-\infty\}\subset A_{K,n}$. On $A_{K,n}$, we have the following lower bound on the probability on the right-hand side of  (\ref{e_condexpLt4}):
\begin{equation} \label{e_X2qK}
P^X_{X_{t-\delta_n}^L} \big( \hat{X}_{\delta_n}([\tau_n + 3\sqrt{\delta_n},\infty)) = 0\,\big) \geq e^{-c_{\ref{th_DIP33}} \, K} =: q_K.
\end{equation}
To prove \eqref{e_X2qK}, first note it is trivial when $\tau_n = -\infty$, because in this case $X^L_{t-\delta_n} = 0$. To see it when $\tau_n > -\infty$ we apply Theorem \ref{th_DIP33}(ii) with $R = 3\sqrt{\delta_n}$ and initial state $X_{t-\delta_n}^L$, along with translation invariance and the change of variables $w=(\tau_n+3\sqrt{\delta_n}-x)/\sqrt{\delta_n}$, to obtain (on $A_{K,n}$)
\begin{align}
&P^X_{X_{t-\delta_n}^L} \big( \hat{X}_{\delta_n}([\tau_n + 3\sqrt{\delta_n},\infty)) = 0\,\big)  \nonumber
\\ &\geq \exp \left(-c_{\ref{th_DIP33}}\, \int_{-\infty}^{\tau_n} (\tau_n + 3 \sqrt{\delta_n} - x)^{-2} \left( \frac{\tau_n + 3 \sqrt{\delta_n} - x}{\sqrt{\delta_n}} \right)^3 \exp \left( -(\tau_n + 3 \sqrt{\delta_n} - x)^2 / 2\delta_n \right) \,X_{t-\delta_n}(x) \,dx  \right) \nonumber
\\ &= \exp \left(-c_{\ref{th_DIP33}}\,\frac{1}{\sqrt{\delta_n}} \int_{3}^{\infty}w\,e^{ -w^2 / 2}\, X_{t-\delta_n}(\tau_n - (w- 3)\sqrt{\delta_n}) \,dw \right)  \nonumber
\\ &\geq e^{-c_{\ref{th_DIP33}}K},\nonumber
\end{align}
which proves (\ref{e_X2qK}), with the final inequality using the fact that $\omega \in A_{K,n}$.  Let $\Lambda_K = \{ A_{K,n} \cap \{\tau_n > -\infty\} \, \text{ infinitely often in } n\}$. That is,
\begin{equation} \label{def_LambdaK}
\Lambda_K = \bigcap_{M=1}^\infty \ \bigcup_{n \geq M} \left( A_{K,n} \cap \{\tau_n > -\infty\}\right).
\end{equation} 
By \eqref{e_condexpLt4} and (\ref{e_X2qK}), for all $\omega \in \Lambda_K$, we have
\[ P^X_{X_0} \big( L_t ([\tau_n+3\sqrt{\delta_n},\infty))> 0 \, \big| \, \cF_{t-\delta_n} \big) \geq 2e^{-2}c_{\ref{lemma_Ltposhalfline}}(3) \, q_K =: p_K\text{ for infinitely many $n$.} \]
It follows that
\begin{equation} \label{e_mtglimsup}
\limsup_{n \to \infty} \, P^X_{X_0} \big( L_t > 0 \, \big| \, \cF_{t-\delta_n} \big)\ge \limsup_{n\to\infty} P^X_{X_0}(L_t([\tau_n+3\sqrt{\delta_n},\infty))>0\big|\cF_{t-\delta_n})\geq p_K\text{ a.s. on }\Lambda_K.
\end{equation}
Moreover, $P^X_{X_0} \big( L_t > 0 \, \big| \, \cF_{t-\delta_n} \big)$ is a bounded martingale and converges almost surely to $P^X_{X_0}\big(L_t>0\, | \, \cF_{t^-} \big)$. Because $s \to X_s$ is a continuous map, we have $X_t = X_{t^-}$ and so $X_t$ is measurable with respect to $\cF_{t^-}$. Moreover, $L_t$ is defined as a  measurable functional of $X_t$ (recall Theorem~A(a)). Thus we have 
\begin{equation}
P^X_{X_0} \big( L_t > 0 \, \big| \, \cF_{t-\delta_n} \big) \to P^X_{X_0}\big(L_t>0\, | \, \cF_{t^-} \big) = 1_{\{L_t>0\}}\,\,\,\, \text{ a.s.}
\end{equation} 
By (\ref{e_mtglimsup}), this implies that $1_{\{L_t > 0\} }(\omega) \geq p_K>0$ a.s. on $\Lambda_K$, and hence
\begin{equation}\label{Lposlam}
L_t>0 \text{ almost surely on $\Lambda_K$.}\\
\end{equation}
The final ingredient of the proof is to show that $\Lambda_K \uparrow \{X_t > 0\}$ as $K \to \infty$ a.s. We proceed by bounding the probability of $A_{K,n}^c$. Using (\ref{e_AKn_def}) and Markov's inequality, we have
\begin{align} \label{e_AKnc_bd1}
P^X_{X_0}(A_{K,n}^c) = P^X_{X_0}(A_{K,n}^c \cap \{\tau_n > -\infty\}) &\leq \frac{1}{K\sqrt{\delta_n}}\,E^X_{X_0} \left(1_{\{\tau_n> -\infty\}}\,\int_3^\infty w e^{-w^2/2}\, X_{t-\delta_n}(\tau_n -(w-3) \sqrt{\delta_n}) \,dw\right).
\end{align}
The first equality follows because $A_{K,n}^c \subseteq \{\tau_n > - \infty\}$. We proceed by integration by parts. For $w > 3$, define $g(w)$ by
\[g(w) = \frac{1}{\sqrt{\delta_n}} X_{t-\delta_n}([\tau_n - (w-3)\sqrt{\delta_n},\tau_n]) =\left(\int_3^w X_{t-\delta_n}( \tau_n - (u-3) \sqrt{\delta_n}) \, du \right). \] 
The last expression, which follows from a change of variables, makes it clear that 
\[g'(w) = X_{t-\delta_n}( \tau_n -(w-3) \sqrt{\delta_n}).\]
Clearly $g(3) = 0$. Taking $f(w) = we^{-w^2/2}$ and proceeding by integration by parts, on $\{\tau_n > -\infty\}$ we have
\begin{align}
\int_3^\infty w e^{-w^2/2}\, X_{t-\delta_n}( \tau_n -(w-3) \sqrt{\delta_n}) \,dw  \nonumber
&= \bigg[ f(w)g(w) \bigg]_3^\infty - \int_3^\infty f'(w) g(w) \, dw \nonumber
\\ &= 0 + \frac{1}{\sqrt{\delta_n}}\int_3^\infty (w^2-1)e^{-w^2/2}X_{t-\delta_n}([\tau_n - (w-3)\sqrt{\delta_n},\tau_n])\, dw. \nonumber
\end{align}
We substitute this into (\ref{e_AKnc_bd1}) and exchange the order of integration (the integrand is positive) to obtain
\begin{equation} \label{e_AKnc_bd2}
P^X_{X_0}(A_{K,n}^c) \leq \frac{1}{\delta_n \, K}\int_3^\infty (w^2-1)e^{-w^2/2} \, P^X_{X_0} \big( 1_{\{\tau_n> -\infty\}}\,X_{t-\delta_n}\big([\tau_n - (w-3)\sqrt{\delta_n} ,\tau_n] \big)\big)\, dw. 
\end{equation}
We now note that the mass term appearing in the integral can be controlled by Proposition \ref{prop_boundaryGrowthBd}. We have for $w\ge 3$
\begin{align}  \label{e_AKnc_bd3}
P^X_{X_0} \big(1_{\{\tau_n> -\infty\}}\, X_{t-\delta_n}(\tau_n - (w-3)\sqrt{\delta_n} ,\tau_n)\big) &\leq P^X_{X_0} \left(\int_{\tau_n - \sqrt{\delta_n}(w-3)}^\infty X_{t-\delta_n}(x)\, dx \right) \nonumber
\\ &\leq c_{\ref{prop_boundaryGrowthBd}}(t/2)X_0(1) \big( \delta_n + ((w-3)\sqrt{\delta_n})^2 \big) \nonumber
\\ &=  c_{\ref{prop_boundaryGrowthBd}}(t/2)X_0(1)\,\delta_n \, (1+(w-3)^2). 
\end{align}
The second inequality is by Proposition \ref{prop_boundaryGrowthBd} and our initial assumption that $\delta_n<t/2$.  Using (\ref{e_AKnc_bd3}) in (\ref{e_AKnc_bd2}), we obtain for $n\ge n_0$,
\begin{align} \label{e_AKnc_bd4}
P^X_{X_0}(A_{K,n}^c) &\leq \frac{c_{\ref{prop_boundaryGrowthBd}}(t/2)X_0(1)}{K} \int_3^\infty (w^2-1) (1+(w-3)^2)e^{-w^2/2} \, dw = \frac{c_0(t,X_0(1))}{K}=:\frac{c_0}{K}.
\end{align}
This allows us to bound $P^X_{X_0}(\{X_t > 0 \} \cap\Lambda_K^c)$ as follows:
\begin{align}
P^X_{X_0}(\{X_t > 0 \} \cap \Lambda_K^c) &= \lim_{M \to \infty} P^X_{X_0} \Bigg( \{X_t>0\} \cap \Bigg[\bigcap_{n=M}^\infty  A_{K,n}^c \cup \{\tau_n = -\infty\}\Bigg] \Bigg) \nonumber
\\ &\leq \lim_{M \to \infty} P^X_{X_0} \Bigg(\{X_t>0\} \cap \Bigg[ \Bigg(  \bigcap_{n \geq M} A_{K,n}^c \Bigg) \bigcup \Bigg( \bigcup_{n=M}^\infty \{\tau_n = - \infty\} \Bigg) \Bigg] \Bigg)  \nonumber
\\ &\leq \lim_{M \to \infty} P^X_{X_0} \Bigg(\Bigg(  \bigcap_{n \geq M} A_{K,n}^c \Bigg) \bigcup \Bigg( \bigcup_{n=M}^\infty \{\tau_n = - \infty\} \cap \{X_t>0\} \Bigg) \Bigg)  \nonumber
\\ &\leq  \Bigl[\lim_{M\to\infty}P^X_{X_0} \big( A_{K,M}^c \big)\Bigr] + P^X_{X_0} \big( \{X_t(1)>0\} \cap \{X_{t-\delta_n}(1) \le \delta_n \,\text{ i.o.} \} \big) \nonumber
\\ &\leq \frac{c_0}{K}.
\end{align}
The second term vanishes because $s \to X_s(1)$ is continuous almost surely, and the bound on the first is by (\ref{e_AKnc_bd4}). We therefore have that
\begin{equation}\label{lamkbig} P^X_{X_0 }\big( \{X_t > 0 \} \backslash \Lambda_K \big) \leq \frac{c_0}{K},
\end{equation}
and hence for $P^X_{X_0}$-almost all $\omega \in \{X_t > 0 \}$, $\omega \in \Lambda_K$ for $K$ sufficiently large. Here we also use the fact that $\Lambda_K$ is increasing in $K$. This and \eqref{Lposlam} completes the proof that $L_t>0$ a.s. on $\{X_t>0\}$.\\

The claim that $L_t((U_t - \delta,U_t)) > 0$ almost surely on $\{X_t > 0 \}$ now follows from two elementary lemmas, the second of which is left as a standard exercise (a variant is known as Hunt's Lemma).
\begin{lemma} \label{lemma_limsupbd}
For all $\delta > 0$, almost surely we have
\begin{equation}
\limsup_{n\to \infty} 1_{\{L_t([\tau_n + 3\sqrt{\delta_n}, \infty)) > 0\}} \leq 1_{\{L_t((U_t - \delta, U_t))>0\}}.
\end{equation}
\end{lemma}
\begin{lemma} \label{lemma_huntlimsup}
Let $(\cF_n)_{n \in \N}$ be an arbitrary filtration, $\cF_\infty$ the minimal $\sigma$-algebra containing $\cF_n$ for all $n$, and let $\{Y_n\}_{n \in \N}$ be a sequence of random variables such that $|Y_n| \leq W$ for all $n\in \N$ for some integrable $W$. Then
\[\limsup_{n} E(Y_n \,| \, \cF_n) \leq E(\limsup_{n} Y_n \, | \, \cF_\infty). \]
\end{lemma}
Lemma~\ref{lemma_limsupbd} is proved at the end of the section. First we see how they complete the proof of Theorem~\ref{thm_01-law} for $P^X_{X_0}$. 
Applying \eqref{e_mtglimsup} and the Lemmas, with $Y_n=1_{\{L_t([\tau_n+3\sqrt{\delta_n},\infty))>0\}}$ in Lemma~\ref{lemma_huntlimsup}, we have
\begin{align*}
1_{\{L_t((U_t-\delta,U_t))>0\}}&\ge\limsup_{n\to\infty}E^X_{X_0}(1_{\{L_t([\tau_n+3\sqrt{\delta_n},\infty))>0}\}|\cF_{t-\delta_n})\\
&\ge p_K\quad\text{a.s. on }\Lambda_K.
\end{align*}
So by \eqref{lamkbig}, $L_t((U_t-\delta,U_t))>0$ $P^X_{X_0}$-a.s. on $\{X_t>0\}$.
 \qed \\

\emph{Proof of Lemma \ref{lemma_limsupbd}.} 
We may work on $\{X_t>0\}$ as both sides are zero if $X_t=0$. By Dominated Convergence we have 
\[\lim_{n\to\infty}\int_{U_t-\delta}^{U_t}X_{t-\delta_n}(x)dx=\int_{U_t-\delta}^{U_t}X_t(x)dx>0\quad\text{on }\{X_t>0\}.\]
This implies that for $n$ large enough, $\int_{U_t-\delta}^{U_t}X_{t-\delta_n}(x)dx>2^{-n}$, and so
\[\tau_n+3\sqrt{\delta_n}>\tau_n>U_t-\delta\quad\text{for $n$ large}.\]
Therefore for $n$ sufficiently large,
\[L_t([\tau_n+3\sqrt{\delta_n},\infty))\le L_t((U_t-\delta,\infty))=L_t((U_t-\delta,U_t)),\]
and the result follows.\qed \\

This completes the proof of Theorem~\ref{thm_01-law} under $P^X_{X_0}$. To see that the same holds under $\N_0$, we apply the above result with $X_0 = \delta_0$. We may assume that $X_t$ is defined by the right-hand side of the cluster decomposition \eqref{e_clusterdecomp1}. So $X_t$ is a sum of $N \sim \text{Poisson}(2/t)$ independent canonical clusters with law $\N_0(X_t \in \cdot \, | \, X_t > 0 )$, and $N=1$ with probability $2t^{-1} e^{-2/t}>0$. In particular we can condition on $N=1$, which gives
\[\N_0 \big( X_t((U_t-\delta,U_t))>0\text{ for all }\delta>0 \, \big| \, X_t > 0 \big) = P^X_{\delta_0} \big(  X_t((U_t-\delta,U_t))>0\text{ for all }\delta>0\, \big| \, N = 1 \big) = 1 \] 
by the result under $P_{\delta_0}^X$ and the inclusion $\{N=1\}\subset\{X_t>0\}$.
Thus the result also holds under $\N_0$. 
\qed

\section{Localization}\label{sec:local01}
In this section we prove Theorem \ref{thm_01_law_local}, which states that $L_t$ has positive mass on any neighborhood of any point in $\partial S(X_t)$ almost surely. The proof uses both decompositions from  Section~\ref{SBMdecomps}, Theorem \ref{thm_01-law}, and the elementary topological fact that if $x \in \partial S(X_t)$, there is a sequence of open ``holes'' in the support near $x$ (Lemma~\ref{lemma_topsupport} below).\\

 Let $d_H$ denote the Hausdorff metric on non-empty compact subsets of $\R$.
That is, $d_H(K_1,K_2)=d_0(K_1,K_2)+d_0(K_2,K_1)$, where $d_0(K_1,K_2)=\inf\{\delta>0:K_1\subset K_2^\delta\}$ and $K_2^\delta$ is the set of points which are less than distance $\delta$ from $K_2$.
\begin{lemma} \label{lemma_topsupport}
$x_0 \in \partial S(X_t)$ if and only if there exists two sequences of non-empty open intervals $I_m$ and $J_m$ such that $d_H(\overline{I_m}, \{x_0\}), d_H(\overline{J_m}, \{x_0\}) \to 0$ as $m \to \infty$, which satisfy $X_t(\cdot)\restrict{I_m} = 0$ and $X_t(\cdot) \restrict{J_m} > 0$ for all $m$.
\end{lemma}
\begin{proof}
Let $x_0 \in \partial S(X_t)$. A sequence $(J_m)_{m=1}^\infty$ with the described conditions must exist because $X_t(\cdot)$ is continuous. We know $B(x_0,2^{-m})\not\subset \overline{\{X_t>0\}}$ because $x_0$ is not an interior point of $\overline{\{X_t>0\}}$. So we may choose an open interval $I_m$ inside the non-empty open set $ B(x_0,2^{-m})\cap\overline{\{X_t>0\}}^c$ which is contained in $B(x_0,2^{-m})\cap\{X_t=0\}$, as required. We leave the converse as an easy exercise. 
\end{proof}

\emph{Proof of Theorem \ref{thm_01_law_local}.} 
We first work under $P_{X_0}^X$ and may assume $X_t=H_t(\{y_t\in\cdot\})$ where
$H$ is an associated historical process. Let $t>0$, $q \in \Q$, and $\delta_n = 2^{-n}$ where we may consider only $\delta_n<t$.   
We again use the decomposition \eqref{histdecomp}, now with $\tau=q$ and $\delta=\delta_n$, that is
\begin{align} \label{e_historicaldecomp}
\hat X^{L,q,\delta_n}_s(\phi) &= \int \phi(y_{t-\delta_n + s})\, 1(y_{t-\delta_n} < q)\, H_{t-\delta_n + s}(dy) \nonumber
\\ \hat{X}^{R,q,\delta_n}_s(\phi) &= \int \phi(y_{t-\delta_n + s})\, 1(y_{t-\delta_n} \geq q)\, H_{t-\delta_n + s}(dy).
\end{align}
As
$\hat X^{L,q,\delta_n}_s + \hat{X}^{R,q,\delta_n}_s = X_{t-\delta_n + s}$, both $\hat X^{L,q,\delta_n}_s$ and $\hat{X}^{R,q,\delta_n}_s$ have densities, which we denote by $\hat X^{L,q,\delta_n}_s(x)$ and $\hat{X}^{R,q,\delta_n}_s(x)$. 
Recall from \eqref{XRdef} that $X^{L,q,\delta_n}_{t-\delta_n}(A) = X_{t-\delta_n}(A \cap (-\infty,q))$ and $X^{R,q,\delta_n}_{t-\delta_n}(A) = X_{t-\delta_n}(A \cap [q,\infty))$ for measurable $A \subseteq \R$. Let $\cF_t$ be the usual right continuous, completed filtration generated by $H$.  By \eqref{histdecomp} and \eqref{histcondlaw}, we have
\begin{align}\label{condlaw}
&X_t=\hat{X}^{R,q,\delta_n}_{\delta_n}+\hat{X}^{L,q,\delta_n}_{\delta_n},\text{ and conditional on $\cF_{t-\delta_n}$, $\hat X^{L,q,\delta_n}$ and $\hat{X}^{R,q,\delta_n}$ are independent}\\
&\text{super-Brownian motions with initial laws $X^{L,q,\delta_n}_{t-\delta_n}$ and $X^{R,q,\delta_n}_{t-\delta_n}$, respectively.}\nonumber
\end{align}
Therefore by Theorem A, for each $s>0$, $\hat{X}_s^{R,q,\delta_n}$ and $\hat{X}_s^{L,q,\delta_n}$ each have a boundary local time, which we denote by $\hat{L}^{R,q,\delta_n}_s$ and $\hat{L}^{L,q,\delta_n}_s$, respectively. By \eqref{condlaw} and applying Theorem~\ref{THMW} conditionally on $\cF_{t-\delta_n}$, we have the following decomposition of $L_t$:
\begin{align} \label{e_localtimedecomp}
L_t(\phi)= \int \phi(x) 1(\hat{X}^{R,q,\delta_n}_{\delta_n}(x) = 0) \, d\hat{L}^{L,q,\delta_n}_{\delta_n}(x) +\int \phi(x) 1(\hat{X}^{L,q,\delta_n}_{\delta_n}(x) = 0) \, d\hat{L}^{R,q,\delta_n}_{\delta_n}(x).
\end{align}
Let $U^{q,\delta_n} = \sup S(\hat{X}^{L,q,\delta_n}_{\delta_n})$.  By \eqref{condlaw} and applying Theorem \ref{thm_01-law} conditionally on $\cF_{t-\delta_n}$, we see that \hfil\break
$\hat{L}^{L,q,\delta_n}_{\delta_n}((U^{q,\delta_n} - \delta, U^{q,\delta_n})) > 0$ for all $\delta > 0$ almost surely on $\{\hat{X}^{L,q,\delta_n}_{\delta_n} > 0\}$. Taking a union over countable events, this implies that
\begin{equation} \label{e_event1as}
\left( \forall \,n \in \N, \forall \, q\in \Q, \hat{X}^{L,q,\delta_n}_{\delta_n}(1) > 0  \Rightarrow \hat{L}^{L,q,\delta_n}_{\delta_n}((U^{q,\delta_n} - \delta, U^{q,\delta_n})) > 0 \,\, \forall \, \delta> 0 \right) \,\,\, P^X_{X_0}\text{-a.s.}
\end{equation}
The fact that $S( \hat{L}^{R,q,\delta_n}_{\delta_n})\subset \partial\{x:\hat{X}^{R,q,\delta_n}_{\delta_n}(x)>0\}$  implies that
\[\hat{X}_{\delta_n}^{R,q,\delta_n}((-\infty,U^{q,\delta_n}))=0\Rightarrow \hat{L}^{R,q,\delta_n}_{\delta_n}((-\infty,U^{q,\delta_n}))=0\text{ a.s.}.\]
Therefore by \eqref{e_localtimedecomp} and \eqref{e_event1as},
\begin{align} \label{e_event2as}
\Bigl( \forall \,n \in \N, \forall \, q \in \Q, \hat{X}^{L,q,\delta_n}_{\delta_n} (1)> 0 \text{ and }&\hat{X}^{R,q,\delta_n}_{\delta_n}((-\infty,U^{q,\delta_n})) = 0 \text{ imply }\\
\nonumber& {L}_t((U^{q,\delta_n} - \delta, U^{q,\delta_n})) =\hat{L}^{L,q,\delta_n}_{\delta_n}((U^{q,\delta_n}-\delta,U^{q,\delta_n}))> 0\ \forall\delta>0\Bigr) \,\,\, P^X_{X_0}\text{-a.s.}
\end{align}
Recall the definitions of $h$ and $K(c,\delta)$ from Section~\ref{SBMdecomps} (see \eqref{Kdef}). By the modulus of continuity for $S(H_t)$, \eqref{Hmod}, 
if $c>2$, then $P^X_{X_0}$-a.a. $\omega$, there exists $\delta = \delta(c,\omega) > 0$ such that $S(H_t) \subset K(c,\delta)$. Thus there exists $\Omega_0 \subset \Omega$ such that $P^X_{X_0}(\Omega_0^c) = 0$ and for all $\omega \in \Omega_0$, the event in (\ref{e_event2as}) and
$S(H_t)\subset K(3,\delta)$ for some $\delta(3,\omega)>0$ both hold.
Let $\omega \in \Omega_0$. Let $a<b$ and suppose that $(a,b) \cap \partial S(X_t) \neq \emptyset$. Then there exists $x_0 \in \partial S(X_t) \cap (a,b)$. By Lemma \ref{lemma_topsupport} there are sequences $(I_m)_{m = 1}^\infty$, $(J_m)_{m=1}^\infty$ of non-empty open intervals converging to $x_0$ with respect to $d_H$ such that $X_t(\cdot)\restrict{I_m} = 0$ and $X_t(\cdot) \restrict{J_m} > 0$. Suppose $I_m = (a_m,b_m)$ and $J_m = (d_m,e_m)$. Since $I_m, J_m \to \{x_0\}$, we can consider $m$ large enough so that $\overline{I_m}, \overline{J_m} \subset (a,b)$. Without loss of generality we assume that $J_m$ lies to the left of $I_m$, ie. that $e_m < a_m$. (If $J_m$ lies to the right of $I_m$, a symmetrical argument dealing with the left-hand endpoints $L^{r,\delta_n}$ of the supports of $X^{r,\delta_n}$ applies). \\

Let $I_m'$ be the open middle third of $I_m$, i.e., $b_m - a_m = l_m$ and $I_m' = (a_m + l_m/3, a_m + 2l_m/3)$. Choose $q \in \Q \cap I_m'$ and $n \in \N$ large enough so that $3h(\delta_n) < l_m / 3$ and $\delta_n < \delta(3,\omega)$. By (\ref{e_historicaldecomp}) and the modulus of continuity,
\begin{equation} 
S(\hat{X}^{R,q,\delta_n}_{\delta_n}) \subseteq (q - 3h(\delta_n), \infty) \subseteq(a_m,\infty), \nonumber
\end{equation}
and hence
\begin{equation} \label{e_nomasscross}
\hat{X}^{R,q,\delta_n}_{\delta_n}((-\infty,a_m]) = 0.
\end{equation}
Moreover, because $(\hat{X}^{L,q,\delta_n}_{\delta_n}+\hat{X}^{R,q,\delta_n}_{\delta_n})(\cdot)=X_t(\cdot)>0$ on $J_m = (d_m, e_m)$, and $e_m < a_m$, we have that 
\begin{equation}\label{Xrdmasspos}
\hat{X}^{L,q,\delta_n}_{\delta_n}((d_m, a_m)) > 0.
\end{equation} 
Furthermore, the modulus of continuity also implies that 
\[S(\hat{X}^{L,q,\delta_n}_{\delta_n}) \subseteq  (-\infty, q+3h(\delta_n)) \subseteq (-\infty, b_m),\]
where the last inclusion holds since $3h(\delta_n)<l_m/3$ and $q\in I_m'$.
The above, together with $X_t((a_m,b_m)) = 0$, implies that $S(\hat{X}^{L,q,\delta_n}_{\delta_n}) \subseteq (-\infty,a_m]$. This and \eqref{Xrdmasspos} imply $U^{q,\delta_n} \in (d_m,a_m] \subset (a,b)$. By (\ref{e_nomasscross}) this implies that $\hat{X}^{R,q,\delta_n}_{\delta_n}((-\infty,U^{q,\delta_n}))=0$ and so by \eqref{Xrdmasspos} we may apply (\ref{e_event2as}) and conclude that  $L_t((U^{q,\delta_n} - \delta, U^{q,\delta_n})) > 0$ for all $\delta >0$. Since $U^{q,\delta_n} \in (d_m,a_m]$, choosing $\delta =  (d_m - a) / 2>0$ (the last by $\overline{J_m}\subset (a,b)$) gives $(U^{q,\delta_n} - \delta, U^{q,\delta_n}) \subset (a,b)$, and hence $L_t((a,b)) > 0$. This proves the result for $P^X_{X_0}$.\\

To see that the same holds under $\N_0$, we proceed as in the proof of Theorem~\ref{thm_01-law} and condition that the Poisson number of clusters, $N$ in \eqref{e_clusterdecomp1}, is one to get
\[\N_0 \big( (a,b) \cap \partial S(X_t) \neq \emptyset \Rightarrow L_t((a,b)) > 0 \, \big| \, X_t > 0 \big) = P^X_{\delta_0} \big( (a,b) \cap \partial S(X_t) \neq \emptyset \Rightarrow L_t((a,b)) > 0 \, \big| \, N = 1 \big) = 1. \] 
Thus, under $\N_0$ the result holds almost surely on $\{X_t > 0 \}$ for all rational $a,b$, and hence holds almost surely.  \qed

\end{document}